\documentclass[12pt]{amsart}
\usepackage{amsfonts}
\usepackage{amssymb}
\usepackage{amsmath}
\usepackage{graphicx}
\usepackage{axodraw}  
\usepackage{amsthm}
\usepackage{amscd}
\usepackage[all]{xy}               
\usepackage{epsfig,exscale}  
\usepackage{color}
\usepackage{txfonts} 
 \font \eightrm=cmr8
 
\newcommand{\nc}{\newcommand}
\newcommand{\butcher}{{\scriptstyle\circleright}}
\newcommand{\lbutcher}{{\raise 3.9pt\hbox{$\circ$}}\hskip -1.9pt{\scriptstyle \searrow}}

\newtheorem{thm}{Theorem}
\newtheorem{exam}{Example}
\newtheorem{exams}[thm]{Examples}
\newtheorem{cor}[thm]{Corollary}

\newtheorem{lem}[thm]{Lemma}
\newtheorem{prop}[thm]{Proposition}
\newtheorem{defn}{Definition}
\newtheorem{rmk}[thm]{Remark}

\nc{\ignore}[1]{{}}
\nc{\mrm}[1]{{\rm #1}}
\nc{\dirlim}{\displaystyle{\lim_{\longrightarrow}}\,}
\nc{\invlim}{\displaystyle{\lim_{\longleftarrow}}\,}
\nc{\vep}{\varepsilon} \nc{\ep}{\epsilon}
\nc{\sigmat}{\widetilde\sigma}
\nc{\ostar}{\overline{*}}

\nc{\mchar}{\mrm{Char}}
\nc{\Hom}{\mrm{Hom}}
\nc{\id}{\mrm{id}}

\nc{\remark}{\noindent{\bf{Remark:}}}
\nc{\remarks}{\noindent{\bf{Remarks:}}}

 \nc{\grad}[1]{^{({#1})}}
 \nc{\fil}[1]{_{#1}}

\nc{\BA}{{\Bbb A}} \nc{\CC}{{\Bbb C}} \nc{\DD}{{\Bbb D}}
\nc{\EE}{{\Bbb E}} \nc{\FF}{{\Bbb F}} \nc{\GG}{{\Bbb G}}
\nc{\HH}{{\Bbb H}} \nc{\LL}{{\Bbb L}} \nc{\NN}{{\Bbb N}}
\nc{\PP}{{\Bbb P}} \nc{\QQ}{{\Bbb Q}} \nc{\RR}{{\Bbb R}}
\nc{\TT}{{\Bbb T}} \nc{\VV}{{\Bbb V}} \nc{\ZZ}{{\Bbb Z}}
\nc{\Cal}[1]{{\mathcal {#1}}}
\nc{\mop}[1]{\mathop{\hbox {\rm #1} }}
\nc{\smop}[1]{\mathop{\hbox {\eightrm #1} }}
\nc{\mopl}[1]{\mathop{\hbox {\rm #1} }\limits}
\nc{\frakg}{{\frak g}}
\nc{\g}[1]{{\frak {#1}}}
\nc{\wt}{\widetilde}
\nc{\wh}{\widehat}
\nc{\un}{\hbox{\bf 1}}
\nc{\redtext}[1]{\textcolor{red}{#1}}
\nc{\bluetext}[1]{\textcolor{blue}{#1}}

\nc\fleche[1]{\mathop{\hbox to #1 mm{\rightarrowfill}}\limits}
\def\semi{\mathrel{\times}\kern -.85pt\joinrel\mathrel{\raise
    1.4pt\hbox{${\scriptscriptstyle |}$}}}

\def\fleche#1{\mathop{\hbox to #1 mm{\rightarrowfill}}\limits}
\def\gfleche#1{\mathop{\hbox to #1 mm{\leftarrowfill}}\limits}
\def\inj#1{\mathop{\hbox to #1 mm{$\lhook\joinrel$\rightarrowfill}}\limits}
\def\ginj#1{\mathop{\hbox to #1 mm{\leftarrowfill$\joinrel\rhook$}}\limits}
\def\surj#1{\mathop{\hbox to #1 mm{\rightarrowfill\hskip 2pt\llap{$\rightarrow$}}}\limits}
\def\gsurj#1{\mathop{\hbox to #1 mm{\rlap{$\leftarrow$}\hskip 2pt
      \leftarrowfill}}\limits}
\def \restr#1{\mathstrut_{\textstyle |}\raise-6pt\hbox{$\scriptstyle #1$}}
\def \srestr#1{\mathstrut_{\scriptstyle |}\hbox to
-1.5pt{}\raise-4pt\hbox{$\scriptscriptstyle #1$}}

\def\diagrama #1{\vskip 4mm \centerline {#1} \vskip 4mm}

\newcommand{\treel}{\hskip 0.5pc\scalebox{-0.3}{{\parbox{0.5pc}{
  \begin{picture}(60,45) (75,10)
    \SetWidth{1.5}
    \SetColor{Black}
    \Line(90,0)(90,45)
  \end{picture}}}}}

\newcommand{\treelsmall}{\hskip 0.5pc\scalebox{-0.3}{{\parbox{0.5pc}{
  \begin{picture}(60,45) (75,20)
    \SetWidth{1.5}
    \SetColor{Black}
    \Line(90,0)(90,45)
  \end{picture}}}}}

 \def\treeLab{{\hskip 0.8pc\scalebox{0.3}{
  \begin{picture}(30,45) (75,-46)
    \SetWidth{1.5}
    \SetColor{Black}
    \Line(90,-50)(60,10)
    \Line(90,-50)(120,10)
    \Line(90,-50)(90,-70)
    \Text(50,15)[lb]{\Huge{\Black{$a$}}}
    \Text(120,15)[lb]{\Huge{\Black{$b$}}}
  \end{picture}}}}

\def\treeLaba{{\hskip 0.5pc\scalebox{0.3}{
  \begin{picture}(30,45) (75,-46)
    \SetWidth{1.5}
    \SetColor{Black}
    \Line(90,-50)(50,20)
    \Line(75,-25)(95,20)
    \Line(130,20)(90,-50)
    \Line(90,-50)(90,-70)
    \Text(45,25)[lb]{\Huge{\Black{$a$}}}
    \Text(90,25)[lb]{\Huge{\Black{$b$}}}
    \Text(130,25)[lb]{\Huge{\Black{$c$}}}
  \end{picture}}}}

\def\treeLabb{{\hskip 0.5pc\scalebox{0.3}{
  \begin{picture}(30,45) (75,-46)
    \SetWidth{1.5}
    \SetColor{Black}
    \Line(90,-50)(50,20)
    \Line(105,-25)(85,20)
    \Line(130,20)(90,-50)
    \Line(90,-50)(90,-70)
    \Text(45,25)[lb]{\Huge{\Black{$a$}}}
    \Text(80,25)[lb]{\Huge{\Black{$b$}}}
    \Text(130,25)[lb]{\Huge{\Black{$c$}}}
  \end{picture}}}}

\def\treeLabc{{\hskip 0.5pc\scalebox{0.3}{
  \begin{picture}(30,45) (75,-46)
    \SetWidth{1.5}
    \SetColor{Black}
    \Line(90,-50)(40,30)
    \Line(75,-25)(93,30)
    \Line(65,-10)(75,30)
    \Line(150,30)(90,-50)
    \Line(120,-10)(110,30)
    \Line(90,-50)(90,-70)
    \Text(35,35)[lb]{\Huge{\Black{$a$}}}
    \Text(70,35)[lb]{\Huge{\Black{$b$}}}
    \Text(88,35)[lb]{\Huge{\Black{$c$}}}
    \Text(108,35)[lb]{\Huge{\Black{$d$}}}
    \Text(150,35)[lb]{\Huge{\Black{$e$}}}
    \Text(60,-38)[lb]{\Huge{\Black{$x$}}}
    \Text(125,-25)[lb]{\Huge{\Black{$y$}}}
  \end{picture}}}}

\newcommand{\treesmall}{\hskip 0.8pc\scalebox{-0.3}{{\parbox{0.5pc}{
  \begin{picture}(30,45) (75,-30)
    \SetWidth{1.5}
    \SetColor{Black}
   \Line(90,0)(70,-40)
    \Line(90,0)(110,-40)
    \Line(90,0)(90,15)
  \end{picture}}}}}

\newcommand{\treeA}{\hskip 1.5pc\scalebox{-0.3}{{\parbox{0.5pc}{
   \begin{picture}(60,75) (75,-30)
    \SetWidth{1.5}
    \SetColor{Black}
    \Line(90,0)(75,-30)
    \Line(90,0)(105,-30)
    \Line(105,30)(90,0)
    \Line(105,30)(135,-30)
    \Line(105,45)(105,30)
  \end{picture}}}}}

\newcommand{\treeB}{\hskip 1.5pc\scalebox{-0.3}{{\parbox{0.5pc}{
  \begin{picture}(60,75) (75,-30)
    \SetWidth{1.5}
    \SetColor{Black}
    \Line(90,0)(75,-30)
    \Line(105,30)(135,-30)
    \Line(105,45)(105,30)
    \Line(120,0)(105,-30)
    \Line(105,30)(90,0)
  \end{picture}}}}}

\newcommand{\treeC}{\hskip 1.5pc\scalebox{-0.2}{{\parbox{0.5pc}{
 \begin{picture}(90,105) (75,-30)
    \SetWidth{2.5}
    \SetColor{Black}
    \Line(90,0)(75,-30)
    \Line(120,60)(90,0)
    \Line(90,0)(105,-30)
    \Line(105,30)(135,-30)
    \Line(120,60)(165,-30)
    \Line(120,80)(120,60)
  \end{picture}
}}}}

\newcommand{\treeD}{\hskip 1.5pc\scalebox{-0.2}{{\parbox{0.5pc}{
 \begin{picture}(90,105) (75,-30)
    \SetWidth{2.5}
    \SetColor{Black}
    \Line(90,0)(75,-30)
    \Line(120,60)(90,0)
    \Line(90,0)(105,-30)
    \Line(120,60)(165,-30)
    \Line(120,80)(120,60)
    \Line(150,0)(135,-30)
  \end{picture}
}}}}

\newcommand{\treeE}{\hskip 1.5pc\scalebox{-0.2}{{\parbox{0.5pc}{
 \begin{picture}(90,105) (75,-30)
    \SetWidth{2.5}
    \SetColor{Black}
    \Line(90,0)(75,-30)
    \Line(120,60)(90,0)
    \Line(120,60)(165,-30)
    \Line(120,80)(120,60)
    \Line(150,0)(135,-30)
    \Line(135,30)(105,-30)
  \end{picture}
}}}}

\newcommand{\treeF}{\hskip 1.5pc\scalebox{-0.2}{{\parbox{0.5pc}{
 \begin{picture}(90,105) (75,-30)
    \SetWidth{2.5}
    \SetColor{Black}
    \Line(90,0)(75,-30)
    \Line(120,60)(90,0)
    \Line(120,60)(165,-30)
    \Line(120,80)(120,60)
    \Line(135,30)(105,-30)
    \Line(120,0)(135,-30)
  \end{picture}
}}}}

\newcommand{\treeG}{\hskip 1.5pc\scalebox{-0.2}{{\parbox{0.5pc}{
 \begin{picture}(90,105) (75,-30)
    \SetWidth{2.5}
    \SetColor{Black}
    \Line(90,0)(75,-30)
    \Line(120,60)(90,0)
    \Line(120,60)(165,-30)
    \Line(120,80)(120,60)
    \Line(105,30)(135,-30)
    \Line(120,0)(105,-30)
  \end{picture}
}}}}

\def\racine{{\scalebox{0.3}{
\begin{picture}(12,12)(38,-38)
\SetWidth{0.5} \SetColor{Black} \Vertex(45,-28){6.5}
\end{picture}}}}

 \def\racinecirc{{\scalebox{0.25}{
  \begin{picture}(34,34) (415,-175)
    \SetWidth{2.0}
    \SetColor{Black}
    \COval(432,-165)(07,07)(0){Black}{White}
  \end{picture}}}}

 \def\arbrea{\,{\scalebox{0.15}{ 
  \begin{picture}(8,55) (370,-248)
    \SetWidth{2}
    \SetColor{Black}
    \Line(374,-244)(374,-200)
    \Vertex(374,-197){9}
    \Vertex(375,-245){12}
  \end{picture}
}}\,}

\def\arbreacirca{\,{\scalebox{0.25}{ 
  \begin{picture}(34,226) (287,-79)
    \SetWidth{2.0}
    \SetColor{Black}
    \COval(290,-35)(7,7)(0){Black}{White}
    \Line(290,-41)(290,-75)
    \COval(290,-75)(07,07)(0){Black}{White}
  \end{picture}
}}\,}

\def\arbreacircb{\,{\scalebox{0.25}{ 
  \begin{picture}(34,226) (287,-79)
    \SetWidth{2.0}
    \SetColor{Black}
    \COval(290,-35)(7,7)(0){Black}{White}
    \Line(290,-41)(290,-75)
    \Vertex(290,-75){8}
\end{picture}
}}\,}

\def\arbreacircc{\,{\scalebox{0.25}{ 
  \begin{picture}(34,226) (287,-79)
    \SetWidth{2.0}
    \SetColor{Black}
    \COval(290,-75)(7,7)(0){Black}{White}
    \Line(290,-41)(290,-68)
    \Vertex(290,-35){8}
\end{picture}
}}\,}

\def\arbreacircd{\,{\scalebox{0.25}{ 
  \begin{picture}(34,226) (287,-79)
    \SetWidth{2.0}
    \SetColor{Black}
    \Vertex(290,-75){8}
    \Line(290,-41)(290,-75)
    \Vertex(290,-35){8}
\end{picture}
}}\,}

 \def\arbreba{\,{\scalebox{0.15}{ 
\begin{picture}(8,106) (370,-197)
    \SetWidth{2}
    \SetColor{Black}
    \Line(374,-193)(374,-149)
    \Vertex(374,-146){9}
    \Vertex(375,-194){12}
    \Line(374,-142)(374,-98)
    \Vertex(374,-95){9}
  \end{picture}
}}\,}

\def\arbrebacirc{\,{\scalebox{0.25}{ 
\begin{picture}(34,178) (335,-95)
    \SetWidth{2.0}
    \SetColor{Black}
    \COval(352,-10)(7,7)(0){Black}{White}
    \Line(352,-16)(352,-50)
    \COval(352,-50)(7,7)(0){Black}{White}
    \Line(352,-56)(352,-90)
    \COval(352,-90)(7,7)(0){Black}{White}
  \end{picture}
}}\,}

\def\arbrebbcirc{\,{\scalebox{0.25}{ 
\begin{picture}(34,178) (335,-95)
    \SetWidth{2.0}
    \SetColor{Black}
    \COval(352,-10)(7,7)(0){Black}{White}
    \Line(352,-16)(352,-50)
    \COval(352,-50)(7,7)(0){Black}{White}
    \Line(352,-56)(352,-90)
    \Vertex(352,-90){8}
  \end{picture}
}}\,}

\def\arbrebccirc{\,{\scalebox{0.25}{ 
\begin{picture}(34,178) (335,-95)
    \SetWidth{2.0}
    \SetColor{Black}
    \COval(352,-10)(7,7)(0){Black}{White}
    \Line(352,-16)(352,-50)
    \Vertex(352,-50){8}
    \Line(352,-56)(352,-90)
    \COval(352,-90)(7,7)(0){Black}{White}
  \end{picture}
}}\,}

\def\arbrebdcirc{\,{\scalebox{0.25}{ 
\begin{picture}(34,178) (335,-95)
    \SetWidth{2.0}
    \SetColor{Black}
    \COval(352,-10)(7,7)(0){Black}{White}
    \Line(352,-16)(352,-50)
    \Vertex(352,-50){8}
    \Line(352,-56)(352,-90)
    \Vertex(352,-90){8}
  \end{picture}
}}\,}

\def\arbrebecirc{\,{\scalebox{0.25}{ 
\begin{picture}(34,178) (335,-95)
    \SetWidth{2.0}
    \SetColor{Black}
    \Vertex(352,-10){8}
    \Line(352,-16)(352,-50)
    \COval(352,-50)(7,7)(0){Black}{White}
    \Line(352,-56)(352,-90)
    \COval(352,-90)(7,7)(0){Black}{White}
  \end{picture}
}}\,}

\def\arbrebfcirc{\,{\scalebox{0.25}{ 
\begin{picture}(34,178) (335,-95)
    \SetWidth{2.0}
    \SetColor{Black}
    \Vertex(352,-10){8}
    \Line(352,-16)(352,-50)
    \COval(352,-50)(7,7)(0){Black}{White}
    \Line(352,-56)(352,-90)
    \Vertex(352,-90){8}
  \end{picture}
}}\,}

\def\arbrebgcirc{\,{\scalebox{0.25}{ 
\begin{picture}(34,178) (335,-95)
    \SetWidth{2.0}
    \SetColor{Black}
    \Vertex(352,-10){8}
    \Line(352,-16)(352,-50)
    \Vertex(352,-50){8}
    \Line(352,-56)(352,-90)
    \COval(352,-90)(7,7)(0){Black}{White}
  \end{picture}
}}\,}

\def\arbrebhcirc{\,{\scalebox{0.25}{ 
\begin{picture}(34,178) (335,-95)
    \SetWidth{2.0}
    \SetColor{Black}
    \Vertex(352,-10){8}
    \Line(352,-16)(352,-50)
    \Vertex(352,-50){8}
    \Line(352,-56)(352,-90)
    \Vertex(352,-90){8}
  \end{picture}
}}\,}

 \def\arbrebb{\,{\scalebox{0.15}{ 
  \begin{picture}(48,48) (349,-255)
    \SetWidth{2}
    \SetColor{Black}
    \Vertex(375,-252){12}
    \Line(376,-250)(395,-215)
    \Line(373,-251)(354,-214)
    \Vertex(353,-211){9}
    \Vertex(395,-213){9}
  \end{picture}
}}}

\def\arbreca{\,{\scalebox{0.15}{
\begin{picture}(8,156) (370,-147)
    \SetWidth{2}
    \SetColor{Black}
    \Line(374,-143)(374,-99)
    \Vertex(374,-96){9}
    \Vertex(375,-144){12}
    \Line(374,-92)(374,-48)
    \Vertex(374,-45){9}
    \Line(374,-42)(374,2)
    \Vertex(374,5){9}
  \end{picture}
}}\,}

\def\arbrecacirc{\,{\scalebox{0.25}{
\begin{picture}(194,142) (287,-131)
    \SetWidth{2.0}
    \SetColor{Black}
    \COval(360,-90)(7,7)(0){Black}{White}
    \COval(410,-90)(7,7)(0){Black}{White}
    \COval(384,-125)(7,7)(0){Black}{White}
    \Line(362,-97)(380,-120)
    \Line(408,-97)(388,-120)
  \end{picture}
}}\,}

\def\arbrecbcirc{\,{\scalebox{0.25}{
\begin{picture}(194,142) (287,-131)
    \SetWidth{2.0}
    \SetColor{Black}
    \COval(360,-90)(7,7)(0){Black}{White}
    \COval(410,-90)(7,7)(0){Black}{White}
    \Vertex(384,-125){8}
    \Line(362,-97)(380,-120)
    \Line(408,-97)(388,-120)
  \end{picture}
}}\,}

\def\arbrecccirc{\,{\scalebox{0.25}{
\begin{picture}(194,142) (287,-131)
    \SetWidth{2.0}
    \SetColor{Black}
    \COval(360,-90)(7,7)(0){Black}{White}
    \Vertex(410,-90){8}
    \COval(384,-125)(7,7)(0){Black}{White}
    \Line(362,-97)(380,-120)
    \Line(408,-97)(388,-120)
  \end{picture}
}}\,}

\def\arbrecdcirc{\,{\scalebox{0.25}{
\begin{picture}(194,142) (287,-131)
    \SetWidth{2.0}
    \SetColor{Black}
    \COval(360,-90)(7,7)(0){Black}{White}
    \Vertex(410,-90){8}
    \Vertex(384,-125){8}
    \Line(362,-97)(380,-120)
    \Line(408,-97)(388,-120)
  \end{picture}
}}\,}

\def\arbrecgcirc{\,{\scalebox{0.25}{
\begin{picture}(194,142) (287,-131)
    \SetWidth{2.0}
    \SetColor{Black}
    \Vertex(360,-90){8}
    \Vertex(410,-90){8}
    \COval(384,-125)(7,7)(0){Black}{White}
    \Line(362,-97)(380,-120)
    \Line(408,-97)(388,-120)
  \end{picture}
}}\,}

\def\arbrechcirc{\,{\scalebox{0.25}{
\begin{picture}(194,142) (287,-131)
    \SetWidth{2.0}
    \SetColor{Black}
    \Vertex(360,-90){8}
    \Vertex(410,-90){8}
    \Vertex(384,-125){8}
    \Line(362,-97)(380,-120)
    \Line(408,-97)(388,-120)
  \end{picture}
}}\,}

\def\arbrecb{\,{\scalebox{0.15}{
\begin{picture}(48,94) (349,-255)
\SetWidth{2}
\SetColor{Black}
\Line(375,-200)(395,-150)
\Line(375,-200)(354,-150)
\Vertex(354,-150){9}
\Vertex(395,-150){9}
\Vertex(375,-200){9}
\Line(375,-250)(375,-200)
\Vertex(375,-250){12}
\end{picture}}}\,}

\def\arbrecc{\,{\scalebox{0.15}{
 \begin{picture}(48,98) (349,-205)
    \SetWidth{2}
    \SetColor{Black}
    \Vertex(375,-202){12}
    \Line(376,-200)(395,-165)
    \Line(373,-201)(354,-164)
    \Vertex(353,-161){9}
    \Vertex(395,-163){9}
    \Line(353,-160)(353,-113)
    \Vertex(353,-111){9}
  \end{picture}
}}\,}

\def\arbreccc{\,{\scalebox{0.15}{
 \begin{picture}(48,98) (349,-205)
    \SetWidth{2}
    \SetColor{Black}
    \Vertex(375,-202){12}
    \Line(376,-200)(395,-165)
    \Line(373,-201)(354,-164)
    \Vertex(353,-161){9}
    \Vertex(395,-163){9}
    \Line(395,-160)(395,-113)
    \Vertex(395,-111){9}
  \end{picture}
}}\,}

\def\arbrecd{\,{\scalebox{0.15}{
\begin{picture}(48,52) (349,-251)
    \SetWidth{2}
    \SetColor{Black}
    \Vertex(375,-248){12}
    \Line(376,-246)(395,-211)
    \Line(373,-247)(354,-210)
    \Vertex(353,-207){9}
    \Vertex(395,-209){9}
    \Line(375,-247)(375,-206)
    \Vertex(376,-203){9}
  \end{picture}
 }}\,}

\def\arbredacirc{\,{\scalebox{0.25}{ 
\begin{picture}(34,178) (335,-95)
    \SetWidth{2.0}
    \SetColor{Black}
   \COval(352,30)(7,7)(0){Black}{White}
    \Line(352,25)(352,-10)
    \COval(352,-10)(7,7)(0){Black}{White}
    \Line(352,-16)(352,-50)
    \COval(352,-50)(7,7)(0){Black}{White}
    \Line(352,-56)(352,-90)
    \COval(352,-90)(7,7)(0){Black}{White}
  \end{picture}
}}\,}

\def\arbredaacirc{\,{\scalebox{0.25}{ 
\begin{picture}(34,178) (335,-95)
    \SetWidth{2.0}
    \SetColor{Black}
    \COval(352,30)(7,7)(0){Black}{White}
    \Line(352,25)(352,-10)
    \COval(352,-10)(7,7)(0){Black}{White}
    \Line(352,-16)(352,-50)
    \COval(352,-50)(7,7)(0){Black}{White}
    \Line(352,-56)(352,-90)
    \Vertex(352,-90){8}
  \end{picture}
}}\,}

\def\arbrebbcirc{\,{\scalebox{0.25}{ 
\begin{picture}(34,178) (335,-95)
    \SetWidth{2.0}
    \SetColor{Black}
    \COval(352,-10)(7,7)(0){Black}{White}
    \Line(352,-16)(352,-50)
    \COval(352,-50)(7,7)(0){Black}{White}
    \Line(352,-56)(352,-90)
    \Vertex(352,-90){8}
  \end{picture}
}}\,}

\def\arbrebccirc{\,{\scalebox{0.25}{ 
\begin{picture}(34,178) (335,-95)
    \SetWidth{2.0}
    \SetColor{Black}
    \COval(352,-10)(7,7)(0){Black}{White}
    \Line(352,-16)(352,-50)
    \Vertex(352,-50){8}
    \Line(352,-56)(352,-90)
    \COval(352,-90)(7,7)(0){Black}{White}
  \end{picture}
}}\,}

\def\arbrebdcirc{\,{\scalebox{0.25}{ 
\begin{picture}(34,178) (335,-95)
    \SetWidth{2.0}
    \SetColor{Black}
    \COval(352,-10)(7,7)(0){Black}{White}
    \Line(352,-16)(352,-50)
    \Vertex(352,-50){8}
    \Line(352,-56)(352,-90)
    \Vertex(352,-90){8}
  \end{picture}
}}\,}

\def\arbrebecirc{\,{\scalebox{0.25}{ 
\begin{picture}(34,178) (335,-95)
    \SetWidth{2.0}
    \SetColor{Black}
    \Vertex(352,-10){8}
    \Line(352,-16)(352,-50)
    \COval(352,-50)(7,7)(0){Black}{White}
    \Line(352,-56)(352,-90)
    \COval(352,-90)(7,7)(0){Black}{White}
  \end{picture}
}}\,}

\def\arbrebfcirc{\,{\scalebox{0.25}{ 
\begin{picture}(34,178) (335,-95)
    \SetWidth{2.0}
    \SetColor{Black}
    \Vertex(352,-10){8}
    \Line(352,-16)(352,-50)
    \COval(352,-50)(7,7)(0){Black}{White}
    \Line(352,-56)(352,-90)
    \Vertex(352,-90){8}
  \end{picture}
}}\,}

\def\arbrebgcirc{\,{\scalebox{0.25}{ 
\begin{picture}(34,178) (335,-95)
    \SetWidth{2.0}
    \SetColor{Black}
    \Vertex(352,-10){8}
    \Line(352,-16)(352,-50)
    \Vertex(352,-50){8}
    \Line(352,-56)(352,-90)
    \COval(352,-90)(7,7)(0){Black}{White}
  \end{picture}
}}\,}

\def\arbredhhcirc{\,{\scalebox{0.25}{ 
\begin{picture}(34,178) (335,-95)
    \SetWidth{2.0}
    \SetColor{Black}
    \Vertex(352,30){8}
    \Line(352,25)(352,-10)
    \Vertex(352,-10){8}
    \Line(352,-16)(352,-50)
    \Vertex(352,-50){8}
    \Line(352,-56)(352,-90)
    \COval(352,-90)(7,7)(0){Black}{White}
  \end{picture}
}}\,}

\def\arbredhhhcirc{\,{\scalebox{0.25}{ 
\begin{picture}(34,178) (335,-95)
    \SetWidth{2.0}
    \Vertex(352,30){8}
    \Line(352,25)(352,-10)
    \SetColor{Black}
    \Vertex(352,-10){8}
    \Line(352,-16)(352,-50)
    \Vertex(352,-50){8}
    \Line(352,-56)(352,-90)
    \Vertex(352,-90){8}
  \end{picture}
}}\,}

\def\arbreeacirc{\,{\scalebox{0.25}{
\begin{picture}(194,142) (287,-131)
    \SetWidth{2.0}
    \SetColor{Black}
    \COval(360,-50)(7,7)(0){Black}{White}
    \COval(410,-50)(7,7)(0){Black}{White}
    \COval(384,-85)(7,7)(0){Black}{White}
    \COval(384,-125)(7,7)(0){Black}{White}
    \Line(384,-90)(384,-120)
    \Line(362,-55)(380,-80)
    \Line(408,-55)(388,-80)
  \end{picture}
}}\,}

\def\arbreeaacirc{\,{\scalebox{0.25}{
\begin{picture}(194,142) (287,-131)
    \SetWidth{2.0}
    \SetColor{Black}
    \COval(360,-50)(7,7)(0){Black}{White}
    \COval(410,-50)(7,7)(0){Black}{White}
    \COval(384,-85)(7,7)(0){Black}{White}
    \Vertex(384,-125){8}
    \Line(384,-90)(384,-120)
    \Line(362,-55)(380,-80)
    \Line(408,-55)(388,-80)
  \end{picture}
}}\,}

\def\arbreehhcirc{\,{\scalebox{0.25}{
\begin{picture}(194,142) (287,-131)
    \SetWidth{2.0}
    \SetColor{Black}
    \Vertex(360,-50){8}
    \Vertex(410,-50){8}
    \Vertex(384,-85){8}
    \COval(384,-125)(7,7)(0){Black}{White}
    \Line(384,-90)(384,-120)
    \Line(362,-55)(380,-80)
    \Line(408,-55)(388,-80)
  \end{picture}
}}\,}

\def\arbreehhhcirc{\,{\scalebox{0.25}{
\begin{picture}(194,142) (287,-131)
    \SetWidth{2.0}
    \SetColor{Black}
    \Vertex(360,-50){8}
    \Vertex(410,-50){8}
    \Vertex(384,-85){8}
    \Vertex(384,-125){8}
    \Line(384,-90)(384,-120)
    \Line(362,-55)(380,-80)
    \Line(408,-55)(388,-80)
  \end{picture}
}}\,}

\def\arbreddcirc{\,{\scalebox{0.25}{ 
\begin{picture}(34,178) (335,-95)
    \SetWidth{2.0}
    \SetColor{Black}
    \COval(352,30)(7,7)(0){Black}{White}
    \Line(352,24)(352,-10)
    \Vertex(352,-10){8}
    \Line(352,-16)(352,-50)
    \COval(352,-50)(7,7)(0){Black}{White}
		\Line(352,-56)(352,-90)
    \Vertex(352,-90){8}
  \end{picture}
}}\,}

\def\arbredfcirc{\,{\scalebox{0.25}{ 
\begin{picture}(34,178) (335,-95)
    \SetWidth{2.0}
    \SetColor{Black}
    \COval(352,30)(7,7)(0){Black}{White}
    \Line(352,24)(352,-10)
    \Vertex(352,-10){8}
    \Line(352,-16)(352,-50)
    \Vertex(352,-50){8}
		\Line(352,-56)(352,-90)
    \COval(352,-90)(7,7)(0){Black}{White}
  \end{picture}
}}\,}

\def\arbredgcirc{\,{\scalebox{0.25}{ 
\begin{picture}(34,178) (335,-95)
    \SetWidth{2.0}
    \SetColor{Black}
    \COval(370,-10)(7,7)(0){Black}{White}
    \Line(370,-16)(352,-50)
    \Vertex(352,-50){8}
    \Line(334,-16)(352,-50)
    \Vertex(352,-90){8}
		\Line(352,-56)(352,-90)
    \COval(334,-10)(7,7)(0){Black}{White}
  \end{picture}
}}\,} 

\def\arbredhcirc{\,{\scalebox{0.25}{ 
\begin{picture}(34,178) (335,-95)
    \SetWidth{2.0}
    \SetColor{Black}
    \COval(334,-10)(7,7)(0){Black}{White}
    \Line(334,-16)(334,-56)
    \Vertex(334,-50){8}
    \Line(370,-56)(352,-90)
    \Vertex(352,-90){8}
		\Line(334,-56)(352,-90)
    \COval(370,-50)(7,7)(0){Black}{White}
  \end{picture}
}}\,}

\def\arbreehcirc{\,{\scalebox{0.25}{ 
\begin{picture}(34,178) (335,-95)
    \SetWidth{2.0}
    \SetColor{Black}
    \COval(334,-10)(7,7)(0){Black}{White}
    \Line(334,-16)(334,-56)
    \COval(334,-50)(7,7)(0){Black}{White}
    \Line(370,-56)(355,-88)
    \COval(352,-90)(7,7)(0){Black}{White}
    \Line(334,-56)(348,-88)
    \COval(370,-50)(7,7)(0){Black}{White}
  \end{picture}
}}\,}

\def\arbrefhcirc{\,{\scalebox{0.25}{ 
\begin{picture}(34,178) (335,-95)
    \SetWidth{2.0}
    \SetColor{Black}
    \COval(327,-50)(7,7)(0){Black}{White}
    \Line(327,-56)(348,-90)
    \COval(352,-50)(7,7)(0){Black}{White}
    \Line(377,-56)(355,-90)
    \COval(352,-90)(7,7)(0){Black}{White}
    \Line(352,-56)(352,-85)
    \COval(377,-50)(7,7)(0){Black}{White}
  \end{picture}
}}\,}

\def\arbreeehcirc{\,{\scalebox{0.25}{ 
\begin{picture}(34,178) (335,-95)
    \SetWidth{2.0}
    \SetColor{Black}
    \COval(334,-10)(7,7)(0){Black}{White}
    \Line(334,-16)(334,-56)
    \COval(334,-50)(7,7)(0){Black}{White}
    \Line(370,-56)(355,-90)
    \Vertex(352,-90){8}
		\Line(334,-56)(348,-90)
    \COval(370,-50)(7,7)(0){Black}{White}
  \end{picture}
}}\,}

\def\arbreeeehcirc{\,{\scalebox{0.25}{ 
\begin{picture}(34,178) (335,-95)
    \SetWidth{2.0}
    \SetColor{Black}
    \Vertex(334,-10){8}
    \Line(334,-16)(334,-56)
    \Vertex(334,-50){8}
    \Line(370,-53)(355,-85)
    \COval(352,-90)(7,7)(0){Black}{White}
   \Line(334,-53)(348,-85)
    \Vertex(370,-50){8}
  \end{picture}
}}\,}

\def\arbreeeeehcirc{\,{\scalebox{0.25}{ 
\begin{picture}(34,178) (335,-95)
    \SetWidth{2.0}
    \SetColor{Black}
    \Vertex(334,-10){8}
    \Line(334,-16)(334,-56)
    \Vertex(334,-50){8}
    \Line(370,-56)(355,-90)
    \Vertex(352,-90){8}
    \Line(334,-56)(348,-90)
    \Vertex(370,-50){8}
  \end{picture}
}}\,}

\def\arbreffhcirc{\,{\scalebox{0.25}{ 
\begin{picture}(34,178) (335,-95)
    \SetWidth{2.0}
    \SetColor{Black}
    \Vertex(327,-50){8}
    \Line(327,-56)(348,-90)
    \Vertex(352,-50){8}
    \Line(377,-56)(355,-90)
    \Vertex(352,-90){8}
    \Line(352,-56)(352,-90)
    \Vertex(377,-50){8}
  \end{picture}
}}\,}

\def\arbredicirc{\,{\scalebox{0.25}{ 
\begin{picture}(34,178) (335,-95)
    \SetWidth{2.0}
    \SetColor{Black}
    \COval(370,-10)(7,7)(0){Black}{White}
    \Line(370,-16)(370,-56)
    \Vertex(370,-50){8}
    \Line(370,-56)(352,-90)
    \Vertex(352,-90){8}
		\Line(334,-56)(352,-90)
    \COval(334,-50)(7,7)(0){Black}{White}
  \end{picture}
}}\,}

\def\arbreda{\,{\scalebox{0.15}{
\begin{picture}(8,204) (370,-99)
    \SetWidth{2}
    \SetColor{Black}
    \Line(374,-95)(374,-51)
    \Vertex(374,-48){9}
    \Vertex(375,-96){12}
    \Line(374,-44)(374,0)
    \Vertex(374,3){9}
    \Line(374,6)(374,50)
    \Vertex(374,53){9}
    \Line(374,53)(374,98)
    \Vertex(374,101){9}
  \end{picture}
}}\,}

\def\arbredb{\,{\scalebox{0.15}{
\begin{picture}(48,135) (349,-255)
    \SetWidth{2}
    \SetColor{Black}
    \Line(376,-150)(395,-100)
    \Line(373,-150)(354,-100)
    \Vertex(353,-100){9}
    \Vertex(395,-100){9}
    \Vertex(374,-150){9}
    \Line(374,-200)(374,-150)
    \Vertex(374,-200){9}
    \Line(374,-250)(374,-200)
    \Vertex(374,-250){12}
  \end{picture}
}}\,}

\def\arbredc{\,{\scalebox{0.15}{
 \begin{picture}(48,150) (349,-205)
    \SetWidth{2}
    \SetColor{Black}
    \Line(376,-148)(395,-113)
    \Line(373,-149)(354,-112)
    \Vertex(353,-109){9}
    \Vertex(395,-111){9}
    \Line(353,-108)(353,-61)
    \Vertex(353,-59){9}
    \Line(374,-200)(374,-153)
    \Vertex(374,-149){9}
    \Vertex(374,-202){12}
  \end{picture}
}}\,}

\def\arbredcc{\,{\scalebox{0.15}{
 \begin{picture}(48,150) (349,-205)
    \SetWidth{2}
    \SetColor{Black}
    \Line(376,-148)(395,-113)
    \Line(373,-149)(354,-112)
    \Vertex(353,-109){9}
    \Vertex(395,-111){9}
    \Line(395,-108)(395,-61)
    \Vertex(395,-59){9}
    \Line(374,-200)(374,-153)
    \Vertex(374,-149){9}
    \Vertex(374,-202){12}
  \end{picture}
}}\,}

\def\arbredd{\,{\scalebox{0.15}{
 \begin{picture}(48,99) (349,-251)
    \SetWidth{2}
    \SetColor{Black}
    \Line(376,-199)(395,-164)
    \Line(373,-200)(354,-163)
    \Vertex(353,-160){9}
    \Vertex(395,-162){9}
    \Vertex(376,-156){9}
    \Vertex(376,-248){12}
    \Line(375,-245)(375,-204)
    \Line(375,-200)(375,-159)
    \Vertex(375,-201){9}
  \end{picture}
}}\,}

\def\arbrede{\,{\scalebox{0.15}{
 \begin{picture}(48,153) (349,-150)
    \SetWidth{2}
    \SetColor{Black}
    \Vertex(375,-147){12}
    \Line(376,-145)(395,-110)
    \Line(373,-146)(354,-109)
    \Vertex(353,-106){9}
    \Vertex(395,-108){9}
    \Line(353,-105)(353,-58)
    \Vertex(353,-56){9}
    \Line(353,-52)(353,-5)
    \Vertex(353,-1){9}
  \end{picture}
}}\,}

\def\arbredee{\,{\scalebox{0.15}{
 \begin{picture}(48,153) (349,-150)
    \SetWidth{2}
    \SetColor{Black}
    \Vertex(375,-147){12}
    \Line(376,-145)(395,-110)
    \Line(373,-146)(354,-109)
    \Vertex(353,-106){9}
    \Vertex(395,-108){9}
    \Line(395,-105)(395,-58)
    \Vertex(395,-56){9}
    \Line(395,-52)(395,-5)
    \Vertex(395,-1){9}
  \end{picture}
}}\,}

\def\arbredf{\,{\scalebox{0.15}{
\begin{picture}(48,98) (349,-205)
    \SetWidth{2}
    \SetColor{Black}
    \Vertex(375,-202){12}
    \Line(376,-200)(395,-165)
    \Line(373,-201)(354,-164)
    \Vertex(353,-161){9}
    \Vertex(395,-163){9}
    \Line(353,-160)(353,-113)
    \Vertex(353,-111){9}
    \Line(395,-159)(395,-112)
    \Vertex(395,-111){9}
  \end{picture}
}}\,}

\def\arbredz{\,{\scalebox{0.15}{
  \begin{picture}(68,88) (329,-215)
    \SetWidth{2}
    \SetColor{Black}
    \Vertex(375,-212){12}
    \Line(376,-210)(395,-175)
    \Line(373,-211)(354,-174)
    \Vertex(353,-171){9}
    \Vertex(395,-173){9}
    \Line(351,-168)(332,-131)
    \Line(355,-168)(374,-133)
    \Vertex(333,-131){9}
    \Vertex(374,-131){9}
  \end{picture}
}}\,}
\def\arbredzz{\,{\scalebox{0.15}{
  \begin{picture}(68,88) (329,-215)
    \SetWidth{2}
    \SetColor{Black}
    \Vertex(375,-212){12}
    \Line(376,-210)(395,-175)
    \Line(373,-211)(354,-174)
    \Vertex(353,-171){9}
    \Vertex(395,-173){9}
    \Line(393,-168)(377,-131)
    \Line(398,-168)(418,-133)
    \Vertex(418,-131){9}
    \Vertex(377,-131){9}
  \end{picture}
}}\,}

\def\arbredg{\,{\scalebox{0.15}{
\begin{picture}(48,98) (349,-205)
    \SetWidth{2}
    \SetColor{Black}
    \Vertex(375,-202){12}
    \Line(376,-200)(395,-165)
    \Line(373,-201)(354,-164)
    \Vertex(353,-161){9}
    \Vertex(395,-163){9}
    \Line(375,-201)(375,-160)
    \Vertex(376,-157){9}
    \Vertex(353,-111){9}
    \Line(353,-161)(353,-111)
  \end{picture}
}}\,}

\def\arbredgg{\,{\scalebox{0.15}{
\begin{picture}(48,98) (349,-205)
    \SetWidth{2}
    \SetColor{Black}
    \Vertex(375,-202){12}
    \Line(376,-200)(395,-165)
    \Line(373,-201)(354,-164)
    \Vertex(353,-161){9}
    \Vertex(395,-163){9}
    \Line(375,-201)(375,-160)
    \Vertex(376,-157){9}
    \Vertex(376,-111){9}
    \Line(375,-155)(375,-114)
  \end{picture}
}}\,}

\def\arbredggg{\,{\scalebox{0.15}{
\begin{picture}(48,98) (349,-205)
    \SetWidth{2}
    \SetColor{Black}
    \Vertex(375,-202){12}
    \Line(376,-200)(395,-165)
    \Line(373,-201)(354,-164)
    \Vertex(353,-161){9}
    \Vertex(395,-163){9}
    \Line(375,-201)(375,-160)
    \Vertex(376,-157){9}
    \Vertex(395,-111){9}
    \Line(395,-163)(395,-111)
  \end{picture}
}}\,}

\def\arbredh{\,{\scalebox{0.15}{
 \begin{picture}(90,46) (330,-257)
    \SetWidth{2}
    \SetColor{Black}
    \Vertex(375,-254){12}
    \Line(376,-252)(395,-217)
    \Vertex(395,-215){9}
    \Line(374,-254)(335,-226)
    \Vertex(334,-224){9}
    \Line(375,-252)(356,-215)
    \Vertex(355,-215){9}
    \Line(374,-255)(417,-227)
    \Vertex(418,-225){9}
  \end{picture}
}}\,}

\def\racineLabr{{\scalebox{0.3}{
\begin{picture}(12,12)(38,-38)
\SetWidth{0.5} \SetColor{Black} \Vertex(45,-28){5.66}
\Text(37,-20)[lb]{\Huge{\Black{$r$}}}
\end{picture}}}}

\def\racineLabell{{\scalebox{0.3}{
\begin{picture}(12,12)(38,-38)
\SetWidth{0.5} \SetColor{Black} \Vertex(45,-28){5.66}
\Text(37,-20)[lb]{\Huge{\Black{$a_{1}^1$}}}
\end{picture}}}}

\def\racineLabeldl{{\scalebox{0.3}{
\begin{picture}(12,12)(38,-38)
\SetWidth{0.5} \SetColor{Black} \Vertex(45,-28){5.66}
\Text(37,-20)[lb]{\Huge{\Black{$a^{1}_{d_1}$}}}
\end{picture}}}}

\def\racineLabeli{{\scalebox{0.3}{
\begin{picture}(12,12)(38,-38)
\SetWidth{0.5} \SetColor{Black} \Vertex(45,-28){5.66}
\Text(37,-20)[lb]{\Huge{\Black{$a_{1}^i$}}}
\end{picture}}}}

\def\racineLabei{{\scalebox{0.3}{
\begin{picture}(12,12)(38,-38)
\SetWidth{0.5} \SetColor{Black} \Vertex(45,-28){5.66}
\Text(37,-20)[lb]{\Huge{\Black{$a
^i$}}}
\end{picture}}}}

\def\racineLabeidi{{\scalebox{0.3}{
\begin{picture}(12,12)(38,-38)
\SetWidth{0.5} \SetColor{Black} \Vertex(45,-28){5.66}
\Text(37,-20)[lb]{\Huge{\Black{$a_{d_i}^i$}}}
\end{picture}}}}

\def\racineLabrj{{\scalebox{0.3}{
\begin{picture}(12,12)(38,-38)
\SetWidth{0.5} \SetColor{Black} \Vertex(45,-28){5.66}
\Text(37,-20)[lb]{\Huge{\Black{$r_s$}}}
\end{picture}}}}

\def\racineLabri{{\scalebox{0.3}{
\begin{picture}(12,12)(38,-38)
\SetWidth{0.5} \SetColor{Black} \Vertex(45,-28){5.66}
\Text(37,-20)[lb]{\Huge{\Black{$r_i$}}}
\end{picture}}}}

\def\racineLabrjj{{\scalebox{0.3}{
\begin{picture}(12,12)(38,-38)
\SetWidth{0.5} \SetColor{Black} \Vertex(45,-28){5.66}
\Text(37,-20)[lb]{\Huge{\Black{$r_{s-1}$}}}
\end{picture}}}}

\def\racineLabrii{{\scalebox{0.3}{
\begin{picture}(12,12)(38,-38)
\SetWidth{0.5} \SetColor{Black} \Vertex(45,-28){5.66}
\Text(37,-20)[lb]{\Huge{\Black{$r_{i-1}$}}}
\end{picture}}}}

\def\racineLabrrr{{\scalebox{0.3}{
\begin{picture}(12,12)(38,-38)
\SetWidth{0.5} \SetColor{Black} \Vertex(45,-28){5.66}
\Text(37,-20)[lb]{\Huge{\Black{$r_1$}}}
\end{picture}}}}

\def\racineLabrrrr{{\scalebox{0.3}{
\begin{picture}(12,12)(38,-38)
\SetWidth{0.5} \SetColor{Black} \Vertex(45,-28){5.66}
\Text(37,-20)[lb]{\Huge{\Black{$r_2$}}}
\end{picture}}}}

\def\racineLabrr{{\scalebox{0.3}{
\begin{picture}(12,12)(38,-38)
\SetWidth{0.5} \SetColor{Black} \Vertex(45,-28){5.66}
\Text(37,-20)[lb]{\Huge{\Black{$r'$}}}
\end{picture}}}}

\def\racineLaba{{\scalebox{0.3}{
\begin{picture}(12,12)(38,-38)
\SetWidth{0.5} \SetColor{Black} \Vertex(45,-28){5.66}
\Text(37,-20)[lb]{\Huge{\Black{$a_1$}}}
\end{picture}}}}

\def\racineLab1{{\scalebox{0.3}{
\begin{picture}(12,12)(38,-38)
\SetWidth{0.5} \SetColor{Black} \Vertex(45,-28){5.66}
\Text(37,-20)[lb]{\Huge{\Black{$1$}}}
\end{picture}}}}

\def\racineLabb{{\scalebox{0.3}{
\begin{picture}(12,12)(38,-38)
\SetWidth{0.5} \SetColor{Black} \Vertex(45,-28){5.66}
\Text(37,-20)[lb]{\Huge{\Black{$2$}}}
\end{picture}}}}

\def\racineLabaii{{\scalebox{0.3}{
\begin{picture}(12,12)(38,-38)
\SetWidth{0.5} \SetColor{Black} \Vertex(45,-28){5.66}
\Text(37,-20)[lb]{\Huge{\Black{$a_2$}}}
\end{picture}}}}

\def\racineLabaa{{\scalebox{0.3}{
\begin{picture}(12,12)(38,-38)
\SetWidth{0.5} \SetColor{Black} \Vertex(45,-28){5.66}
\Text(37,-20)[lb]{\Huge{\Black{$a$}}}
\end{picture}}}}

\def\racineLabooo{{\scalebox{0.3}{
\begin{picture}(12,12)(38,-38)
\SetWidth{0.5} \SetColor{Black} \Vertex(45,-28){5.66}
\Text(37,-20)[lb]{\Huge{\Black{$a_3$}}}
\end{picture}}}}

\def\racineLaboooo{{\scalebox{0.3}{
\begin{picture}(12,12)(38,-38)
\SetWidth{0.5} \SetColor{Black} \Vertex(45,-28){5.66}
\Text(37,-20)[lb]{\Huge{\Black{$a_4$}}}
\end{picture}}}}

\def\racineLabi{{\scalebox{0.3}{
\begin{picture}(12,12)(38,-38)
\SetWidth{0.5} \SetColor{Black} \Vertex(45,-28){5.66}
\Text(37,-20)[lb]{\Huge{\Black{$i$}}}
\end{picture}}}}

\def\arbreaLaba{\,{\scalebox{0.2}{ 
  \begin{picture}(8,55) (370,-248)
    \SetWidth{2}
    \SetColor{Black}
    \Line(374,-244)(374,-200)
    \Vertex(374,-197){9}
    \Vertex(375,-245){12}
    \Text(370,-285)[lb]{\Huge{\Black{$1$}}}
    \Text(365,-180)[lb]{\Huge{\Black{$1$}}}
  \end{picture}
}}\,}

\def\arbreaoLaba{\,{\scalebox{0.25}{ 
  \begin{picture}(8,55) (370,-248)
    \SetWidth{2}
    \SetColor{Black}
    \Line(374,-244)(374,-200)
    \Vertex(374,-197){9}
    \Vertex(375,-245){12}
    \Text(370,-285)[lb]{\Huge{\Black{$a_1$}}}
    \Text(365,-180)[lb]{\Huge{\Black{$a_1$}}}
  \end{picture}
}}\,}

\def\arbreaLabto{\,{\scalebox{0.25}{ 
  \begin{picture}(8,55) (370,-248)
    \SetWidth{2}
    \SetColor{Black}
    \Line(374,-244)(374,-200)
    \Vertex(374,-197){9}
    \Vertex(375,-245){12}
    \Text(375,-280)[lb]{\Huge{\Black{$a_3$}}}
    \Text(365,-180)[lb]{\Huge{\Black{$a_1$}}}
  \end{picture}
}}\,}

\def\arbreaLabot{\,{\scalebox{0.25}{ 
  \begin{picture}(8,55) (370,-248)
    \SetWidth{2}
    \SetColor{Black}
    \Line(374,-244)(374,-200)
    \Vertex(374,-197){9}
    \Vertex(375,-245){12}
    \Text(375,-280)[lb]{\Huge{\Black{$a_1$}}}
    \Text(365,-180)[lb]{\Huge{\Black{$a_3$}}}
  \end{picture}
}}\,}

\def\arbreaLabtt{\,{\scalebox{0.25}{ 
  \begin{picture}(8,55) (370,-248)
    \SetWidth{2}
    \SetColor{Black}
    \Line(374,-244)(374,-200)
    \Vertex(374,-197){9}
    \Vertex(375,-245){12}
    \Text(375,-285)[lb]{\Huge{\Black{$a_2$}}}
    \Text(365,-180)[lb]{\Huge{\Black{$a_2$}}}
  \end{picture}
}}\,}

\def\arbreaLaboo{\,{\scalebox{0.25}{ 
  \begin{picture}(8,55) (370,-248)
    \SetWidth{2}
    \SetColor{Black}
    \Line(374,-244)(374,-200)
    \Vertex(374,-197){9}
    \Vertex(375,-245){12}
    \Text(375,-280)[lb]{\Huge{\Black{$a_1$}}}
    \Text(365,-180)[lb]{\Huge{\Black{$a_2$}}}
  \end{picture}
}}\,}

\def\arbreaLabooo{\,{\scalebox{0.25}{ 
  \begin{picture}(8,55) (370,-248)
    \SetWidth{2}
    \SetColor{Black}
    \Line(374,-244)(374,-200)
    \Vertex(374,-197){9}
    \Vertex(375,-245){12}
    \Text(375,-280)[lb]{\Huge{\Black{$a_2$}}}
    \Text(365,-180)[lb]{\Huge{\Black{$a_1$}}}
  \end{picture}
}}\,}

\def\arbreaaLab{\,{\scalebox{0.2}{ 
  \begin{picture}(8,55) (370,-248)
    \SetWidth{2}
    \SetColor{Black}
    \Line(364,-244)(384,-200)
    \Vertex(384,-197){9}
    \Vertex(365,-245){9}
    \Text(360,-275)[lb]{\Huge{\Black{$v$}}}
    \Text(380,-185)[lb]{\Huge{\Black{$w$}}}
  \end{picture}
}}\,}

\def\arbreabLab{\,{\scalebox{0.2}{ 
  \begin{picture}(8,55) (370,-248)
    \SetWidth{2}
    \SetColor{Black}
    \Line(364,-244)(384,-200)
    \Vertex(384,-197){9}
    \Vertex(365,-245){9}
    \Text(360,-275)[lb]{\Huge{\Black{$u$}}}
    \Text(380,-185)[lb]{\Huge{\Black{$w'$}}}
  \end{picture}
}}\,}

\def\arbreaLabb{\,{\scalebox{0.2}{ 
  \begin{picture}(8,55) (370,-248)
    \SetWidth{2}
    \SetColor{Black}
    \Line(374,-244)(374,-200)
    \Vertex(374,-197){9}
    \Vertex(375,-245){12}
    \Text(370,-285)[lb]{\Huge{\Black{$2$}}}
    \Text(365,-180)[lb]{\Huge{\Black{$1$}}}
  \end{picture}
}}\,}
\def\arbreaLabc{\,{\scalebox{0.2}{ 
  \begin{picture}(8,55) (370,-248)
    \SetWidth{2}
    \SetColor{Black}
    \Line(374,-244)(374,-200)
    \Vertex(374,-197){9}
    \Vertex(375,-245){12}
    \Text(370,-285)[lb]{\Huge{\Black{$1$}}}
    \Text(365,-180)[lb]{\Huge{\Black{$2$}}}
  \end{picture}
}}\,}

\def\arbreaLabd{\,{\scalebox{0.2}{ 
  \begin{picture}(8,55) (370,-248)
    \SetWidth{2}
    \SetColor{Black}
    \Line(374,-244)(374,-200)
    \Vertex(374,-197){9}
    \Vertex(375,-245){12}
    \Text(370,-285)[lb]{\Huge{\Black{$2$}}}
    \Text(365,-180)[lb]{\Huge{\Black{$2$}}}
  \end{picture}
}}\,}

\def\arbreaLabe{\,{\scalebox{0.2}{ 
  \begin{picture}(8,55) (370,-248)
    \SetWidth{2}
    \SetColor{Black}
    \Line(374,-244)(374,-200)
    \Vertex(374,-197){9}
    \Vertex(375,-245){12}
    \Text(370,-285)[lb]{\Huge{\Black{$3$}}}
    \Text(365,-180)[lb]{\Huge{\Black{$1$}}}
  \end{picture}
}}\,}

\def\arbreaLabf{\,{\scalebox{0.2}{ 
  \begin{picture}(8,55) (370,-248)
    \SetWidth{2}
    \SetColor{Black}
    \Line(374,-244)(374,-200)
    \Vertex(374,-197){9}
    \Vertex(375,-245){12}
    \Text(370,-285)[lb]{\Huge{\Black{$1$}}}
    \Text(365,-180)[lb]{\Huge{\Black{$3$}}}
  \end{picture}
}}\,}

 \def\arbrebaLaba{\,{\scalebox{0.2}{ 
\begin{picture}(8,106) (370,-197)
    \SetWidth{2}
    \SetColor{Black}
    \Line(374,-193)(374,-149)
    \Vertex(374,-146){9}
    \Vertex(375,-194){12}
    \Line(374,-142)(374,-98)
    \Vertex(374,-95){9}
    \Text(370,-233)[lb]{\Huge{\Black{$1$}}}
    \Text(390,-155)[lb]{\Huge{\Black{$1$}}}
    \Text(370,-80)[lb]{\Huge{\Black{$1$}}}
  \end{picture}
}}\,}

 \def\arbrebaLabo{\,{\scalebox{0.25}{ 
\begin{picture}(8,106) (370,-197)
    \SetWidth{2}
    \SetColor{Black}
    \Line(374,-193)(374,-149)
    \Vertex(374,-146){9}
    \Vertex(375,-194){12}
    \Line(374,-142)(374,-98)
    \Vertex(374,-95){9}
    \Text(370,-235)[lb]{\Huge{\Black{$a_1$}}}
    \Text(390,-155)[lb]{\Huge{\Black{$a_1$}}}
    \Text(370,-80)[lb]{\Huge{\Black{$a_1$}}}
  \end{picture}
}}\,}

 \def\arbrebaLaboot{\,{\scalebox{0.25}{ 
\begin{picture}(8,106) (370,-197)
    \SetWidth{2}
    \SetColor{Black}
    \Line(374,-193)(374,-149)
    \Vertex(374,-146){9}
    \Vertex(375,-194){12}
    \Line(374,-142)(374,-98)
    \Vertex(374,-95){9}
    \Text(370,-235)[lb]{\Huge{\Black{$a_2$}}}
    \Text(390,-155)[lb]{\Huge{\Black{$a_1$}}}
    \Text(370,-80)[lb]{\Huge{\Black{$a_1$}}}
  \end{picture}
}}\,}

 \def\arbrebaLaboto{\,{\scalebox{0.25}{ 
\begin{picture}(8,106) (370,-197)
    \SetWidth{2}
    \SetColor{Black}
    \Line(374,-193)(374,-149)
    \Vertex(374,-146){9}
    \Vertex(375,-194){12}
    \Line(374,-142)(374,-98)
    \Vertex(374,-95){9}
    \Text(370,-235)[lb]{\Huge{\Black{$a_1$}}}
    \Text(390,-155)[lb]{\Huge{\Black{$a_1$}}}
    \Text(370,-80)[lb]{\Huge{\Black{$a_2$}}}
  \end{picture}
}}\,}

 \def\arbrebaLabtoo{\,{\scalebox{0.25}{ 
\begin{picture}(8,106) (370,-197)
    \SetWidth{2}
    \SetColor{Black}
    \Line(374,-193)(374,-149)
    \Vertex(374,-146){9}
    \Vertex(375,-194){12}
    \Line(374,-142)(374,-98)
    \Vertex(374,-95){9}
    \Text(370,-235)[lb]{\Huge{\Black{$a_1$}}}
    \Text(390,-155)[lb]{\Huge{\Black{$a_2$}}}
    \Text(370,-80)[lb]{\Huge{\Black{$a_1$}}}
  \end{picture}
}}\,}

\def\arbrebaLaboi{\,{\scalebox{0.25}{ 
\begin{picture}(8,106) (370,-197)
    \SetWidth{2}
    \SetColor{Black}
    \Line(374,-193)(374,-149)
    \Vertex(374,-146){9}
    \Vertex(375,-194){12}
    \Line(374,-142)(374,-98)
    \Vertex(374,-95){9}
    \Text(370,-235)[lb]{\Huge{\Black{$a_1$}}}
    \Text(390,-155)[lb]{\Huge{\Black{$a_1$}}}
    \Text(370,-80)[lb]{\Huge{\Black{$a_1$}}}
  \end{picture}
}}\,}

\def\arbrebaLabb{\,{\scalebox{0.2}{ 
\begin{picture}(8,106) (370,-197)
    \SetWidth{2}
    \SetColor{Black}
    \Line(374,-193)(374,-149)
    \Vertex(374,-146){9}
    \Vertex(375,-194){12}
    \Line(374,-142)(374,-98)
    \Vertex(374,-95){9}
    \Text(370,-233)[lb]{\Huge{\Black{$2$}}}
    \Text(390,-155)[lb]{\Huge{\Black{$1$}}}
    \Text(370,-80)[lb]{\Huge{\Black{$1$}}}
  \end{picture}
}}\,}

\def\arbrebaLabc{\,{\scalebox{0.2}{ 
\begin{picture}(8,106) (370,-197)
    \SetWidth{2}
    \SetColor{Black}
    \Line(374,-193)(374,-149)
    \Vertex(374,-146){9}
    \Vertex(375,-194){12}
    \Line(374,-142)(374,-98)
    \Vertex(374,-95){9}
    \Text(370,-233)[lb]{\Huge{\Black{$1$}}}
    \Text(390,-155)[lb]{\Huge{\Black{$2$}}}
    \Text(370,-80)[lb]{\Huge{\Black{$1$}}}
  \end{picture}
}}\,}
\def\arbrebaLabd{\,{\scalebox{0.2}{ 
\begin{picture}(8,106) (370,-197)
    \SetWidth{2}
    \SetColor{Black}
    \Line(374,-193)(374,-149)
    \Vertex(374,-146){9}
    \Vertex(375,-194){12}
    \Line(374,-142)(374,-98)
    \Vertex(374,-95){9}
    \Text(370,-233)[lb]{\Huge{\Black{$1$}}}
    \Text(390,-155)[lb]{\Huge{\Black{$1$}}}
    \Text(370,-80)[lb]{\Huge{\Black{$2$}}}
  \end{picture}
}}\,}
\def\arbrebaLabe{\,{\scalebox{0.2}{ 
\begin{picture}(8,106) (370,-197)
    \SetWidth{2}
    \SetColor{Black}
    \Line(374,-193)(374,-149)
    \Vertex(374,-146){9}
    \Vertex(375,-194){12}
    \Line(374,-142)(374,-98)
    \Vertex(374,-95){9}
    \Text(370,-233)[lb]{\Huge{\Black{$2$}}}
    \Text(390,-155)[lb]{\Huge{\Black{$2$}}}
    \Text(370,-80)[lb]{\Huge{\Black{$1$}}}
  \end{picture}
}}\,}
\def\arbrebaLabf{\,{\scalebox{0.2}{ 
\begin{picture}(8,106) (370,-197)
    \SetWidth{2}
    \SetColor{Black}
    \Line(374,-193)(374,-149)
    \Vertex(374,-146){9}
    \Vertex(375,-194){12}
    \Line(374,-142)(374,-98)
    \Vertex(374,-95){9}
    \Text(370,-233)[lb]{\Huge{\Black{$2$}}}
    \Text(390,-155)[lb]{\Huge{\Black{$1$}}}
    \Text(370,-80)[lb]{\Huge{\Black{$2$}}}
  \end{picture}
}}\,}
\def\arbrebaLabg{\,{\scalebox{0.2}{ 
\begin{picture}(8,106) (370,-197)
    \SetWidth{2}
    \SetColor{Black}
    \Line(374,-193)(374,-149)
    \Vertex(374,-146){9}
    \Vertex(375,-194){12}
    \Line(374,-142)(374,-98)
    \Vertex(374,-95){9}
    \Text(370,-233)[lb]{\Huge{\Black{$1$}}}
    \Text(390,-155)[lb]{\Huge{\Black{$2$}}}
    \Text(370,-80)[lb]{\Huge{\Black{$2$}}}
  \end{picture}
}}\,}
\def\arbrebaLabh{\,{\scalebox{0.2}{ 
\begin{picture}(8,106) (370,-197)
    \SetWidth{2}
    \SetColor{Black}
    \Line(374,-193)(374,-149)
    \Vertex(374,-146){9}
    \Vertex(375,-194){12}
    \Line(374,-142)(374,-98)
    \Vertex(374,-95){9}
    \Text(370,-233)[lb]{\Huge{\Black{$2$}}}
    \Text(390,-155)[lb]{\Huge{\Black{$2$}}}
    \Text(370,-80)[lb]{\Huge{\Black{$2$}}}
  \end{picture}
}}\,}
 
\def\arbrebbLaba{\,{\scalebox{0.2}{
  \begin{picture}(48,48) (349,-255)
    \SetWidth{2}
    \SetColor{Black}
    \Vertex(375,-252){12}
    \Line(376,-250)(395,-215)
    \Line(373,-251)(354,-214)
    \Vertex(353,-211){9}
    \Vertex(395,-211){9}
    \Text(370,-285)[lb]{\Huge{\Black{$1$}}}
    \Text(320,-215)[lb]{\Huge{\Black{$1$}}}
    \Text(410,-215)[lb]{\Huge{\Black{$1$}}}
  \end{picture}
}}}

\def\arbrebbLabo{\,{\scalebox{0.25}{
  \begin{picture}(48,48) (349,-255)
    \SetWidth{2}
    \SetColor{Black}
    \Vertex(375,-252){12}
    \Line(376,-250)(395,-215)
    \Line(373,-251)(354,-214)
    \Vertex(353,-211){9}
    \Vertex(395,-211){9}
    \Text(370,-285)[lb]{\Huge{\Black{$a_1$}}}
    \Text(320,-215)[lb]{\Huge{\Black{$a_1$}}}
    \Text(410,-215)[lb]{\Huge{\Black{$a_1$}}}
  \end{picture}
}}}

\def\arbrebbLabtoo{\,{\scalebox{0.25}{
  \begin{picture}(48,48) (349,-255)
    \SetWidth{2}
    \SetColor{Black}
    \Vertex(375,-252){12}
    \Line(376,-250)(395,-215)
    \Line(373,-251)(354,-214)
    \Vertex(353,-211){9}
    \Vertex(395,-211){9}
    \Text(370,-285)[lb]{\Huge{\Black{$a_2$}}}
    \Text(320,-210)[lb]{\Huge{\Black{$a_1$}}}
    \Text(410,-210)[lb]{\Huge{\Black{$a_1$}}}
  \end{picture}
}}}

\def\arbrebbLaboto{\,{\scalebox{0.25}{
  \begin{picture}(48,48) (349,-255)
    \SetWidth{2}
    \SetColor{Black}
    \Vertex(375,-252){12}
    \Line(376,-250)(395,-215)
    \Line(373,-251)(354,-214)
    \Vertex(353,-211){9}
    \Vertex(395,-211){9}
    \Text(370,-285)[lb]{\Huge{\Black{$a_1$}}}
    \Text(320,-210)[lb]{\Huge{\Black{$a_1$}}}
    \Text(410,-210)[lb]{\Huge{\Black{$a_2$}}}
  \end{picture}
}}}

\def\arbrebbLabb{\,{\scalebox{0.2}{
  \begin{picture}(48,48) (349,-255)
    \SetWidth{2}
    \SetColor{Black}
    \Vertex(375,-252){12}
    \Line(376,-250)(395,-215)
    \Line(373,-251)(354,-214)
    \Vertex(353,-211){9}
    \Vertex(395,-211){9}
    \Text(370,-285)[lb]{\Huge{\Black{$2$}}}
    \Text(320,-215)[lb]{\Huge{\Black{$1$}}}
    \Text(410,-215)[lb]{\Huge{\Black{$1$}}}
  \end{picture}
}}}
\def\arbrebbLabc{\,{\scalebox{0.2}{
  \begin{picture}(48,48) (349,-255)
    \SetWidth{2}
    \SetColor{Black}
    \Vertex(375,-252){12}
    \Line(376,-250)(395,-215)
    \Line(373,-251)(354,-214)
    \Vertex(353,-211){9}
    \Vertex(395,-211){9}
    \Text(370,-285)[lb]{\Huge{\Black{$1$}}}
    \Text(320,-215)[lb]{\Huge{\Black{$1$}}}
    \Text(410,-215)[lb]{\Huge{\Black{$2$}}}
  \end{picture}
}}}
\def\arbrebbLabd{\,{\scalebox{0.2}{
  \begin{picture}(48,48) (349,-255)
    \SetWidth{2}
    \SetColor{Black}
    \Vertex(375,-252){12}
    \Line(376,-250)(395,-215)
    \Line(373,-251)(354,-214)
    \Vertex(353,-211){9}
    \Vertex(395,-211){9}
    \Text(370,-285)[lb]{\Huge{\Black{$1$}}}
    \Text(320,-215)[lb]{\Huge{\Black{$2$}}}
    \Text(410,-215)[lb]{\Huge{\Black{$1$}}}
  \end{picture}
}}}
\def\arbrebbLabae{\,{\scalebox{0.2}{
  \begin{picture}(48,48) (349,-255)
    \SetWidth{2}
    \SetColor{Black}
    \Vertex(375,-252){12}
    \Line(376,-250)(395,-215)
    \Line(373,-251)(354,-214)
    \Vertex(353,-211){9}
    \Vertex(395,-211){9}
    \Text(370,-285)[lb]{\Huge{\Black{$2$}}}
    \Text(320,-215)[lb]{\Huge{\Black{$1$}}}
    \Text(410,-215)[lb]{\Huge{\Black{$2$}}}
  \end{picture}
}}}
\def\arbrebbLabf{\,{\scalebox{0.2}{
  \begin{picture}(48,48) (349,-255)
    \SetWidth{2}
    \SetColor{Black}
    \Vertex(375,-252){12}
    \Line(376,-250)(395,-215)
    \Line(373,-251)(354,-214)
    \Vertex(353,-211){9}
    \Vertex(395,-211){9}
    \Text(370,-285)[lb]{\Huge{\Black{$2$}}}
    \Text(320,-215)[lb]{\Huge{\Black{$2$}}}
    \Text(410,-215)[lb]{\Huge{\Black{$1$}}}
  \end{picture}
}}}
\def\arbrebbLabg{\,{\scalebox{0.2}{
  \begin{picture}(48,48) (349,-255)
    \SetWidth{2}
    \SetColor{Black}
    \Vertex(375,-252){12}
    \Line(376,-250)(395,-215)
    \Line(373,-251)(354,-214)
    \Vertex(353,-211){9}
    \Vertex(395,-211){9}
    \Text(370,-285)[lb]{\Huge{\Black{$1$}}}
    \Text(320,-215)[lb]{\Huge{\Black{$2$}}}
    \Text(410,-215)[lb]{\Huge{\Black{$2$}}}
  \end{picture}
}}}
\def\arbrebbLabh{\,{\scalebox{0.2}{
  \begin{picture}(48,48) (349,-255)
    \SetWidth{2}
    \SetColor{Black}
    \Vertex(375,-252){12}
    \Line(376,-250)(395,-215)
    \Line(373,-251)(354,-214)
    \Vertex(353,-211){9}
    \Vertex(395,-211){9}
    \Text(370,-285)[lb]{\Huge{\Black{$2$}}}
    \Text(320,-215)[lb]{\Huge{\Black{$2$}}}
    \Text(410,-215)[lb]{\Huge{\Black{$2$}}}
  \end{picture}
}}}

\def\arbrelmlLab{\,{\scalebox{0.3}{
  \begin{picture}(48,48) (349,-255)
    \SetWidth{2}
    \SetColor{Black}
    \Vertex(375,-252){5}
    \Line(376,-250)(395,-215)
    \Line(373,-251)(354,-214)
    \Vertex(353,-211){4}
    \Vertex(395,-211){4}
    \Text(370,-282)[lb]{\Huge{\Black{$r$}}}
    \Text(330,-215)[lb]{\Huge{\Black{$l_{1}$}}}
    \Text(405,-215)[lb]{\Huge{\Black{$l_{2}$}}}
  \end{picture}
}}}

\def\arbrelplLab{\,{\scalebox{0.3}{
  \begin{picture}(48,48) (349,-255)
    \SetWidth{2}
    \SetColor{Black}
    \Vertex(375,-252){5}
    \Line(376,-250)(395,-215)
    \Line(373,-251)(354,-214)
    \Vertex(353,-211){4}
    \Vertex(395,-211){4}
    \Text(370,-282)[lb]{\Huge{\Black{$r$}}}
    \Text(330,-215)[lb]{\Huge{\Black{$l_{2}$}}}
    \Text(405,-215)[lb]{\Huge{\Black{$l_{1}$}}}
  \end{picture}
}}}

\def\arbrelmlpLab{\,{\scalebox{0.3}{
  \begin{picture}(48,48) (349,-255)
    \SetWidth{2}
    \SetColor{Black}
    \Vertex(375,-252){5}
    \Line(375,-250)(375,-215)
    \Line(375,-211)(375,-175)
    \Vertex(375,-211){4}
    \Vertex(375,-170){4}
    \Text(370,-282)[lb]{\Huge{\Black{$r$}}}
    \Text(385,-180)[lb]{\Huge{\Black{$l_{1}$}}}
    \Text(385,-220)[lb]{\Huge{\Black{$l_{2}$}}}
  \end{picture}
}}}

\def\arbrelmrLab{\,{\scalebox{0.3}{
  \begin{picture}(48,48) (349,-255)
    \SetWidth{2}
    \SetColor{Black}
    \Vertex(375,-252){5}
    \Line(375,-250)(375,-215)
    \Line(375,-211)(375,-175)
    \Vertex(375,-211){4}
    \Vertex(375,-170){4}
    \Text(385,-275)[lb]{\Huge{\Black{$l_{2}$}}}
    \Text(385,-180)[lb]{\Huge{\Black{$l_{1}$}}}
    \Text(385,-220)[lb]{\Huge{\Black{$r$}}}
  \end{picture}
}}}

\def\arbrelmrlpLab{\,{\scalebox{0.3}{
  \begin{picture}(48,48) (349,-255)
    \SetWidth{2}
    \SetColor{Black}
    \Vertex(375,-252){5}
    \Line(375,-253)(355,-211)
    \Line(375,-252)(395,-211)
    \Line(355,-211)(355,-175)
    \Vertex(395,-211){4}
    \Vertex(355,-211){4}
    \Vertex(355,-170){4}
    \Text(375,-280)[lb]{\Huge{\Black{$r_1$}}}
    \Text(405,-220)[lb]{\Huge{\Black{$l_{2}^{(1)}$}}}
    \Text(363,-180)[lb]{\Huge{\Black{$l_{1}$}}}
    \Text(363,-215)[lb]{\Huge{\Black{$r$}}}
  \end{picture}
}}}

\def\arbrelmrslp{\,{\scalebox{0.3}{
\begin{picture}(48,48) (349,-255)

    \SetWidth{2}
    \SetColor{Black}
    \Vertex(375,-252){5}
    \Line(375,-253)(355,-211)
    \Line(375,-252)(395,-211)
    \DashLine(355,-211)(355,-175){4}
    \Line(355,-130)(355,-175)
    \Vertex(355,-211){4}
    \Vertex(355,-130){4}
    \Vertex(395,-211){4}
    \Vertex(355,-170){4}
    \Text(375,-280)[lb]{\Huge{\Black{$r_s$}}}
    \Text(405,-220)[lb]{\Huge{\Black{$l_{2}^{(s)}$}}}
    \Text(335,-180)[lb]{\Huge{\Black{$r$}}}
    \Text(335,-140)[lb]{\Huge{\Black{$l_{1}$}}}
    \Text(310,-215)[lb]{\Huge{\Black{$r_{s-1}$}}}
  \end{picture}
}}}

\def\arbrelplmrs{\,{\scalebox{0.3}{
\begin{picture}(48,48) (349,-255)

    \SetWidth{2}
    \SetColor{Black}
    \Vertex(375,-252){5}
    \Line(375,-253)(355,-211)
    \Line(375,-252)(395,-211)
    \DashLine(395,-211)(395,-175){4}
    \Line(395,-130)(395,-175)
    \Vertex(355,-211){4}
    \Vertex(395,-130){4}
    \Vertex(395,-211){4}
    \Vertex(395,-170){4}
    \Text(375,-280)[lb]{\Huge{\Black{$r_s$}}}
    \Text(405,-220)[lb]{\Huge{\Black{$r_{s-1}$}}}
    \Text(405,-180)[lb]{\Huge{\Black{$r$}}}
		\Text(405,-140)[lb]{\Huge{\Black{$l_{1}$}}}
    \Text(330,-215)[lb]{\Huge{\Black{$l_{2}^{(s)}$}}}
  \end{picture}
}}}

\def\arbrelmrlprLab{\,{\scalebox{0.3}{
  \begin{picture}(48,48) (349,-255)
    \SetWidth{2}
    \SetColor{Black}
    \Vertex(375,-252){5}
    \Line(375,-253)(375,-211)
    \Line(375,-211)(375,-175)
    \Line(375,-130)(375,-175)
    \Vertex(375,-211){4}
    \Vertex(375,-130){4}
    \Vertex(375,-170){4}
    \Text(375,-280)[lb]{\Huge{\Black{$r_1$}}}
    \Text(383,-225)[lb]{\Huge{\Black{$l_{2}^{(1)}$}}}
    \Text(383,-140)[lb]{\Huge{\Black{$l_{1}$}}}
    \Text(383,-175)[lb]{\Huge{\Black{$r$}}}
  \end{picture}
}}}

\def\arbrelmrrlpLab{\,{\scalebox{0.3}{
  \begin{picture}(48,48) (349,-255)
    \SetWidth{2}
    \SetColor{Black}
    \Vertex(375,-252){5}
    \Line(375,-253)(375,-211)
    \Line(375,-211)(375,-175)
    \Line(375,-130)(375,-175)
    \Vertex(375,-211){4}
    \Vertex(375,-130){4}
    \Vertex(375,-170){4}
    \Text(383,-275)[lb]{\Huge{\Black{$l_{2}^{(1)}$}}}
    \Text(383,-225)[lb]{\Huge{\Black{$r_1$}}}
    \Text(383,-140)[lb]{\Huge{\Black{$l_{1}$}}}
    \Text(383,-175)[lb]{\Huge{\Black{$r$}}}
  \end{picture}
}}}

\def\arbrelplmrrLab{\,{\scalebox{0.3}{
  \begin{picture}(48,48) (349,-255)
    \SetWidth{2}
    \SetColor{Black}
    \Vertex(375,-252){5}
    \Line(375,-253)(375,-211)
    \Line(376,-211)(395,-170)
    \Line(374,-211)(355,-170)
    \Vertex(395,-170){4}
    \Vertex(375,-211){4}
    \Vertex(355,-170){4}
    \Text(375,-280)[lb]{\Huge{\Black{$r_1$}}}
    \Text(395,-168)[lb]{\Huge{\Black{$l_{1}$}}}
    \Text(355,-173)[lb]{\Huge{\Black{$l_{2}^{(1)}$}}}
    \Text(385,-215)[lb]{\Huge{\Black{$r$}}}
  \end{picture}
}}}

\def\arbrelprlmLab{\,{\scalebox{0.3}{
  \begin{picture}(48,48) (349,-255)
    \SetWidth{2}
    \SetColor{Black}
    \Vertex(375,-252){5}
    \Line(375,-253)(355,-211)
    \Line(375,-252)(395,-211)
    \Line(395,-211)(395,-175)
    \Vertex(395,-211){4}
    \Vertex(355,-211){4}
    \Vertex(395,-170){4}
    \Text(375,-280)[lb]{\Huge{\Black{$r_1$}}}
    \Text(405,-220)[lb]{\Huge{\Black{$r$}}}
    \Text(405,-180)[lb]{\Huge{\Black{$l_{1}$}}}
    \Text(355,-215)[lb]{\Huge{\Black{$l_{2}^{(1)}$}}}
  \end{picture}
}}}

\def\arbrelplmrrrLab{\,{\scalebox{0.3}{
  \begin{picture}(48,48) (349,-255)
    \SetWidth{2}
    \SetColor{Black}
    \Vertex(375,-252){5}
    \Line(375,-253)(355,-211)
    \Line(375,-252)(395,-211)
    \Line(395,-211)(395,-175)
		\Line(395,-130)(395,-175)
    \Vertex(395,-211){4}
		\Vertex(395,-130){4}
    \Vertex(355,-211){4}
    \Vertex(395,-170){4}
    \Text(375,-280)[lb]{\Huge{\Black{$r_2$}}}
    \Text(405,-220)[lb]{\Huge{\Black{$r_1$}}}
    \Text(405,-180)[lb]{\Huge{\Black{$r$}}}
		\Text(405,-140)[lb]{\Huge{\Black{$l_{1}$}}}
    \Text(353,-215)[lb]{\Huge{\Black{$l_{2}^{(2)}$}}}
  \end{picture}
}}}

\def\arbrelmrrlpsLab{\,{\scalebox{0.3}{
  \begin{picture}(48,48) (349,-255)
    \SetWidth{2}
    \SetColor{Black}
    \Vertex(375,-252){5}
    \Line(375,-211)(355,-170)
    \Line(375,-252)(375,-211)
    \DashLine(375,-211)(395,-173){4}
    \Line(395,-130)(395,-173)
    \Vertex(375,-211){4}
    \Vertex(395,-130){4}
    \Vertex(355,-170){4}
    \Vertex(395,-170){4}
    \Text(375,-280)[lb]{\Huge{\Black{$r_s$}}}
    \Text(385,-225)[lb]{\Huge{\Black{$r_{s-1}$}}}
    \Text(405,-180)[lb]{\Huge{\Black{$r$}}}
		\Text(405,-140)[lb]{\Huge{\Black{$l_{1}$}}}
    \Text(345,-165)[lb]{\Huge{\Black{$l_{2}^{(s)}$}}}
  \end{picture}
}}}

\def\arbrelplmrrrsLab{\,{\scalebox{0.3}{
  \begin{picture}(48,48) (349,-255)
    \SetWidth{2}
    \SetColor{Black}
    \Vertex(375,-252){5}
    \Line(375,-130)(375,-90)
    \Line(375,-252)(375,-211)
    \Line(375,-211)(375,-175)
    \DashLine(375,-130)(375,-175){4}
    \Vertex(375,-211){4}
    \Vertex(375,-130){4}
    \Vertex(375,-90){4}
    \Vertex(375,-170){4}
    \Text(375,-280)[lb]{\Huge{\Black{$r_s$}}}
    \Text(385,-230)[lb]{\Huge{\Black{$l_{2}^{(s)}$}}}
    \Text(385,-180)[lb]{\Huge{\Black{$r_{s-1}$}}}
		\Text(385,-140)[lb]{\Huge{\Black{$r$}}}
		\Text(385,-105)[lb]{\Huge{\Black{$l_{1}$}}}
    \end{picture}
}}}

\def\arbrelplmLab{\,{\scalebox{0.3}{
  \begin{picture}(48,48) (349,-255)
    \SetWidth{2}
    \SetColor{Black}
    \Vertex(375,-252){5}
    \Line(375,-250)(375,-215)
    \Line(375,-211)(375,-175)
    \Vertex(375,-211){4}
    \Vertex(375,-170){4}
    \Text(370,-282)[lb]{\Huge{\Black{$r$}}}
    \Text(385,-180)[lb]{\Huge{\Black{$l_{2}$}}}
    \Text(385,-220)[lb]{\Huge{\Black{$l_{1}$}}}
  \end{picture}
}}}

\def\arbreLabgencasel{\,{\scalebox{0.3}{
  \begin{picture}(48,48) (349,-255)
    \SetWidth{2}
    \SetColor{Black}
    \Vertex(375,-252){5}
    \Line(380,-250)(440,-214)
    \Line(370,-251)(320,-214)
    \Line(373,-251)(365,-214)
    \Line(376,-251)(393,-214)
    \Vertex(320,-211){4}
    \Vertex(440,-211){4}
    \Vertex(365,-211){4}
    \Vertex(393,-211){4}
    \Text(370,-275)[lb]{\Huge{\Black{$r$}}}
    \Text(310,-205)[lb]{\Huge{\Black{$\sigma_1$}}}
    \Text(355,-205)[lb]{\Huge{\Black{$\sigma_{_{\!\!i-1}}$}}}
    \Text(390,-205)[lb]{\Huge{\Black{$\sigma_i$}}}
    \Text(330,-213)[lb]{\Huge{\Black{$\ldots$}}}
    \Text(407,-213)[lb]{\Huge{\Black{$\ldots$}}}
    \Text(435,-205)[lb]{\Huge{\Black{$\sigma_k$}}}
  \end{picture}
}}}

\def\arbreLabcaseji{\,{\scalebox{0.3}{
  \begin{picture}(48,48) (349,-255)
    \SetWidth{2}
    \SetColor{Black}
    \Vertex(375,-252){5}
    \Line(380,-250)(430,-214)
    \Line(370,-251)(340,-214)
    \Line(375,-251)(385,-214)
    \Line(385,-211)(385,-170)
    \Vertex(340,-211){4}
    \Vertex(430,-211){4}
    \Vertex(385,-211){4}
    \Vertex(385,-170){4}
    \Text(370,-275)[lb]{\Huge{\Black{$r$}}}
    \Text(330,-205)[lb]{\Huge{\Black{$\sigma_1$}}}
    \Text(392,-208)[lb]{\Huge{\Black{$\sigma_{i}$}}}
    \Text(392,-180)[lb]{\Huge{\Black{$\sigma_{_{\!\!i-1}}$}}}
    \Text(350,-213)[lb]{\Huge{\Black{$\ldots$}}}
    \Text(395,-213)[lb]{\Huge{\Black{$\ldots$}}}
    \Text(425,-205)[lb]{\Huge{\Black{$\sigma_k$}}}
  \end{picture}
}}}

\def\arbreLabcaseij{\,{\scalebox{0.3}{
  \begin{picture}(48,48) (349,-255)
    \SetWidth{2}
    \SetColor{Black}
    \Vertex(375,-252){5}
    \Line(380,-250)(430,-214)
    \Line(370,-251)(340,-214)
    \Line(375,-251)(385,-214)
    \Line(385,-211)(385,-170)
    \Vertex(340,-211){4}
    \Vertex(430,-211){4}
    \Vertex(385,-211){4}
    \Vertex(385,-170){4}
    \Text(370,-275)[lb]{\Huge{\Black{$r$}}}
    \Text(330,-205)[lb]{\Huge{\Black{$\sigma_1$}}}
    \Text(391,-211)[lb]{\Huge{\Black{$\sigma_{_{\!\!i-1}}$}}}
    \Text(392,-180)[lb]{\Huge{\Black{$\sigma_{i}$}}}
    \Text(350,-213)[lb]{\Huge{\Black{$\ldots$}}}
    \Text(395,-213)[lb]{\Huge{\Black{$\ldots$}}}
    \Text(428,-205)[lb]{\Huge{\Black{$\sigma_k$}}}
  \end{picture}
}}}

\def\arbreLabgencasell{\,{\scalebox{0.3}{
  \begin{picture}(48,48) (349,-255)
    \SetWidth{2}
    \SetColor{Black}
    \Vertex(375,-252){5}
    \Line(380,-250)(440,-214)
    \Line(370,-251)(320,-214)
    \Line(373,-251)(365,-214)
    \Line(376,-251)(393,-214)
    \Vertex(320,-211){4}
    \Vertex(440,-211){4}
    \Vertex(365,-211){4}
    \Vertex(393,-211){4}
    \Text(370,-275)[lb]{\Huge{\Black{$r$}}}
    \Text(310,-205)[lb]{\Huge{\Black{$\sigma_1$}}}
    \Text(355,-205)[lb]{\Huge{\Black{$\sigma_{i}$}}}
    \Text(387,-205)[lb]{\Huge{\Black{$\sigma_{_{\!\!i-1}}$}}}
    \Text(330,-213)[lb]{\Huge{\Black{$\ldots$}}}
    \Text(407,-213)[lb]{\Huge{\Black{$\ldots$}}}
    \Text(435,-205)[lb]{\Huge{\Black{$\sigma_k$}}}
  \end{picture}
}}}

\def\arbrexyLabb{\,{\scalebox{0.25}{
  \begin{picture}(48,48) (349,-255)
    \SetWidth{2}
    \SetColor{Black}
    \Vertex(375,-252){12}
    \Line(376,-250)(395,-215)
    \Line(373,-251)(354,-214)
    \Vertex(353,-211){9}
    \Vertex(395,-211){9}
    \Text(370,-282)[lb]{\Huge{\Black{$r$}}}
    \Text(330,-215)[lb]{\Huge{\Black{$x$}}}
    \Text(410,-215)[lb]{\Huge{\Black{$y$}}}
  \end{picture}
}}}

\def\arbrexylLab{\,{\scalebox{0.25}{
  \begin{picture}(48,48) (349,-255)
    \SetWidth{2}
    \SetColor{Black}
    \Vertex(375,-252){12}
    \Line(375,-215)(375,-178)
    \Line(375,-252)(375,-214)
    \Vertex(375,-214){9}
    \Vertex(375,-175){9}
    \Text(370,-282)[lb]{\Huge{\Black{$x$}}}
    \Text(355,-218)[lb]{\Huge{\Black{$r$}}}
    \Text(370,-165)[lb]{\Huge{\Black{$y$}}}
  \end{picture}
}}}

\def\arbrecaLabo{\,{\scalebox{0.25}{
\begin{picture}(8,156) (370,-147)
    \SetWidth{2}
    \SetColor{Black}
    \Line(374,-143)(374,-99)
    \Vertex(374,-96){9}
    \Vertex(375,-144){12}
    \Line(374,-92)(374,-48)
    \Vertex(374,-45){9}
    \Line(374,-42)(374,2)
    \Vertex(374,5){9}
    \Text(370,-185)[lb]{\Huge{\Black{$a_1$}}}
    \Text(370,15)[lb]{\Huge{\Black{$a_1$}}}
    \Text(390,-55)[lb]{\Huge{\Black{$a_1$}}}
    \Text(390,-105)[lb]{\Huge{\Black{$a_1$}}}
  \end{picture}
}}}

\def\arbrecbLabo{\,{\scalebox{0.25}{
\begin{picture}(48,94) (349,-255)
\SetWidth{2}
\SetColor{Black}
\Line(375,-200)(395,-150)
\Line(375,-200)(354,-150)
\Vertex(354,-150){9}
\Vertex(395,-150){9}
\Vertex(375,-200){9}
\Line(375,-250)(375,-200)
\Vertex(375,-250){12}
\Text(370,-290)[lb]{\Huge{\Black{$a_1$}}}
\Text(343,-135)[lb]{\Huge{\Black{$a_1$}}}
\Text(388,-135)[lb]{\Huge{\Black{$a_1$}}}
\Text(390,-210)[lb]{\Huge{\Black{$a_1$}}}
\end{picture}}}\,}

\def\arbreccLabo{\,{\scalebox{0.25}{
 \begin{picture}(48,98) (349,-205)
    \SetWidth{2}
    \SetColor{Black}
    \Vertex(375,-202){12}
    \Line(376,-200)(395,-165)
    \Line(373,-201)(354,-164)
    \Vertex(353,-161){9}
    \Vertex(395,-163){9}
    \Line(353,-160)(353,-113)
    \Vertex(353,-111){9}
    \Text(370,-245)[lb]{\Huge{\Black{$a_1$}}}
    \Text(315,-170)[lb]{\Huge{\Black{$a_1$}}}
    \Text(410,-170)[lb]{\Huge{\Black{$a_1$}}}
    \Text(345,-100)[lb]{\Huge{\Black{$a_1$}}}
  \end{picture}
}}\,}

\def\arbrecorLabo{\,{\scalebox{0.25}{
 \begin{picture}(48,98) (349,-205)
    \SetWidth{2}
    \SetColor{Black}
    \Vertex(375,-202){12}
    \Line(376,-200)(415,-165)
    \Line(375,-200)(375,-165)
    \Vertex(375,-163){9}
    \Vertex(415,-163){9}
    \Vertex(335,-163){9}
    \Line(375,-200)(335,-163)
    \Text(370,-245)[lb]{\Huge{\Black{$a_1$}}}
    \Text(330,-150)[lb]{\Huge{\Black{$a_1$}}}
    \Text(410,-150)[lb]{\Huge{\Black{$a_1$}}}
    \Text(372,-150)[lb]{\Huge{\Black{$a_1$}}}
  \end{picture}
}}\,}

\begin{document}
\title[Monomial bases]
      {Monomial bases and pre-Lie structure for free Lie algebras}
      
\author{Mahdi J. Hasan Al-Kaabi}
\author{Dominique Manchon}
\author{Fr\'ed\'eric Patras}

\address {Mathematics Department, College of Science, Al-Mustansiriya University, Palestine Street, P.O.Box 14022, Baghdad, IRAQ. Mahdi.Alkaabi@math.univ-bpclermont.fr}  
\address{LMBP, CNRS-UMR6620, Universit\'e Clermont-Auvergne, 3 place Vasar\'ely, CS 60026, F63178 Aubi\`ere, CEDEX, France. Dominique.Manchon@math.univ-bpclermont.fr}
\address{Laboratoire J.A.Dieudonn\'e, UMR CNRS-UNS N°7351
 Universit\'e de Nice Sophia-Antipolis, Parc Valrose, 06108 NICE Cedex 2. Frederic.PATRAS@unice.fr}

\date{}
\begin{abstract}
In this paper, we construct a pre-Lie structure on the free Lie algebra $\Cal{L}(E)$ generated by a set $E$, giving an explicit presentation of $\Cal{L}(E)$ as the quotient of the free pre-Lie algebra $\Cal{T}^E$, generated by the (non-planar) $E$-decorated rooted trees, by some ideal $I$. The main result in this paper is a description of Gr\"obner bases in terms of trees.    
\end{abstract}

\maketitle
\tableofcontents
\textbf{Math. Subject Classification:} 05C05, 17D25, 17A50, 17B01.

\textbf{Keywords:} pre-Lie algebras, monomial bases, free Lie algebras, rooted trees.

\section{Introduction}
In the spirit of Felix Klein's (1849-1925) "Erlangen Program", any Lie group $G$\label{G} is a group of symmetries of some class of differentiable manifolds. The corresponding infinitesimal transformations are given by the Lie algebra of $G$, which is the set of left-invariant vector fields on $G$. The problem of classification of groups of transformations has been considered by S. Lie (1842-1899) not only for subgroups of $GL_n$, but also for infinite dimensional groups \cite{VK90}. \\

 The problem of classification of simple finite-dimensional Lie algebras over the field of complex numbers was solved at the end of the 19th century by W. Killing (1847-1923) and E. Cartan (1869-1951). The central figure of the origins of the theory of the structure of Lie algebras is W. Killing, whose paper in four parts laid the conceptual foundations of the theory. In 1884, Killing introduced the concept of Lie algebra independently of Lie and formulated the problem of determining all possible structures for real, finite dimensional Lie algebras. The joint work of Killing and Cartan establishes the foundations of the theory. Killing's work contained many gaps which Cartan succeeded to fill \cite{TH82}, \cite{VK90}. W. Killing, H. Cartan, S. Lie, and F. Engel are the main authors of the early development of the theory and some of its various applications.\\

The concept of pre-Lie algebras appeared in many works under various names. E. B. Vinberg and M. Gerstenhaber in 1963 independently presented  the concept under two different names; "right symmetric algebras" and "pre-Lie algebras" respectively \cite{E.V63, M.G63}. Other denominations, e.g. "Vinberg algebras", appeared since then. "Chronological algebras" is the term used by A. Agrachev and R. V. Gamkrelidze in their work on \textsl{nonstationary vector fields} \cite{AR81}. The term "pre-Lie algebras" is now the standard terminology. The Lie algebra of a real connected Lie group $G$ admits a compatible pre-Lie structure if and only if $G$ admits a left-invariant affine structure \cite[Proposition 2.31]{D.B06}, see also the work of J. L. Koszul \cite{J.K61} for more details about the pre-Lie structure, in a geometrical point of view.\\

In Sections \ref{section-one}, \ref{section-two} of this paper, we recall some basics: trees, Lie and pre-Lie algebras, Gr\"obner bases. We construct, in Section \ref{CIIIsecI}, a structure of pre-Lie algebra on the free Lie algebra $\Cal{L}(E)$ generated by a set $E$, and we give the explicit presentation of $\Cal{L}(E)$ as the quotient of the free pre-Lie algebra $\Cal{T}^E$ by some ideal.\\

Recall that $\Cal{T}^E_{\!\!pl}$ is the linear span of the set $T^E_{\!pl}$ of all planar $E$-decorated rooted trees, which forms together with the left Butcher product $\lbutcher$, and the left grafting $\searrow$\label{sea2} respectively two magmatic algebras. In Section \ref{mwo}, we give a tree version of a monomial well-order on $T^E_{\!pl}$. We adapt the work of T. Mora \cite{T.M} on Gr\"obner bases to a non-associative, magmatic context, using the descriptions of the free magmatic algebras $\big(\Cal{T}^E_{\!\!pl}, \lbutcher \big)$ and $\big(\Cal{T}^E_{\!\!pl}, \searrow\big)$ respectively, following \cite{D.H.}. We split the basis of $E$- decorated planar rooted trees into two parts $O(J')$ and $T(J')$, where $J'$ is the ideal of $\Cal{T}^E_{\!\!pl}$ generated by the pre-Lie identity and by weighted anti-symmetry relations:
$$|\sigma| \sigma \lbutcher \tau + |\tau| \tau \lbutcher \sigma.$$
\noindent Here $T(J')$ is the set of maximal terms of elements of $J'$, and its complement $O(J')$ then defines a basis of $\Cal{L}(E)$. We get one of the important results (Theorem \ref{main}), on the description of the set $O(J')$ in terms of trees.\\ 

In Section \ref{CIIIsecIII}, we give a non-planar tree version of the monomial well-order above. We describe monomial bases for the pre-Lie (respectively free Lie) algebra $\Cal{L}(E)$, using the procedures of Gr\"obner bases and our work described in \cite{AM2014}, in the monomial basis for the free pre-Lie algebra $\Cal{T}^E$.  

\section{Trees}\label{section-one}
In graph theory, a tree is a undirected connected finite graph, without cycles \cite{EM}. A rooted tree is defined as a tree with one designated vertex called the root. The other remaining vertices are partitioned into $k\geq 0$ disjoint subsets such that each of them in turn represents a rooted tree, and a subtree of the whole tree. This can be taken as a recursive definition for rooted trees, widely used in computer algorithms \cite{D.K68}. Rooted trees stand among the most important structures appearing in many branches of pure and applied mathematics. \\

In general, a tree structure can be described as a "branching" relationship between vertices, much like that found in the trees of nature. Many types of trees defined by all sorts of constraints on properties of vertices appear to be of interest in combinatorics and in related areas such as formal logic and computer science.\\

A planar binary tree is a finite oriented tree embedded in the plane, such that each internal vertex has exactly two incoming edges and one outgoing edge. One of the internal vertices, called the root, is a distinguished vertex with two incoming edges and one edge, like a tail at the bottom, not ending at a vertex. The incoming edges in this type of trees are internal (connecting two internal vertices), or external (with one free end). The external incoming edges are called the leaves. We give here some examples of planar binary trees:
$$\treel\,\,\,\,\treesmall\,\,\,\,\,\,\treeA\,\,\,\,\treeB\,\,\,\,\treeC\,\,\,\,\treeE\,\,\,\,\treeD\,\,\,\,\treeF\,\,\,\,\treeG\,\,\,\,\ldots,$$
where the single edge "$\treelsmall$" is the unique planar binary tree without internal vertices. The degree of any planar binary tree is the number of its leaves. Denote by $T^{bin}_{pl}$ (respectively $\Cal{T}^{bin}_{\!pl}$) the set (respectively the linear span) of planar binary trees.\\

Define the grafting operation "$\vee$" on the space $\Cal{T}^{bin}_{\!pl}$\label{cbin1} to be the operation that maps any planar binary trees $t_1, t_2$ into a new planar binary tree $t_1 \vee t_2$, which takes the $Y$-shaped tree $\treesmall$ replacing the left (respectively the right) branch by $t_1$ (respectively $t_2$), see the following examples:
$$ \treel \vee \treel = \treesmall\,,\,\,\,\treel \vee \treesmall = \treeA\,,\,\,\,\treesmall \vee \treel = \treeB\,,\,\,\,\treesmall \vee \treesmall = \treeD\,,\,\,\,\treel \vee \treeA = \treeC.$$

The number of binary trees of degree $n$ is given by the Catalan number $c_n = \frac{(2n)!}{(n+1)! n!}$, where the first ones are $1, 1, 2, 5, 14, 42, 132, \ldots$. This sequence of numbers is the sequence A000108 in \cite{S}. \\

Let $E$\label{E3} be a (non-empty) set. The free magma $M(E)$ generated by $E$ can be described as the set of planar binary trees with leaves decorated by the elements of $E$, together with the "$\vee$" product described above \cite{D.K68, PT09}. Moreover, the linear span $\Cal{T}^{bin, E}_{\!\!\!pl}$, generated by the trees of the magma $M(E)= T^{bin, E}_{\!\!pl}$ defined above, equipped with the grafting "$\vee$" is a description of the free magmatic algebra.\\

For any positive integer $n$, a rooted tree of degree $n$, or simply $n$-rooted tree, is a finite oriented tree together with $n$ vertices. One of them, called the root, is a distinguished vertex without any outgoing edge. Any vertex can have arbitrarily many incoming edges, and any vertex distinct from the root has exactly one outgoing edge. Vertices with no incoming edges are called leaves.\\

A rooted tree is said to be planar, if it is endowed with an embedding in the plane. Otherwise, its called a (non-planar) rooted tree. Let $E$ be a (non-empty) set. An $E$-decorated rooted tree is a pair $(t, d)$ of a rooted tree $t$ together with a map $d: V(t) \rightarrow E$, which decorates each vertex $v$ of $t$ by an element $a$ of $E$, i.e. $d(v) = a$, where $V(t)$ is the set of all vertices of $t$. Here are the planar (undecorated) rooted trees up to five vertices:
$$\racine\hskip 6mm\arbrea\hskip 6mm \arbreba\,\arbrebb\hskip 6mm\arbreca\,\arbrecb\,\arbrecc\,\arbreccc\,\arbrecd\hskip 6mm \arbreda\,\arbredb\,\arbredc\,\arbredcc\,\arbredd\,\arbrede\,\arbredee\,\arbredf\,\arbredz\!\arbredzz\;\arbredg\,\arbredgg  \arbredggg\,\arbredh\,\,\,\,\cdots $$

From now on, we will consider that all our trees are decorated, except for some cases in which we will state the property explicitly. Denote by $T^{E}_{\!pl}$\label{plE1} (respectively $T^{E}$) the set of all planar (respectively non-planar) decorated rooted trees, and $\Cal{T}^{E}_{\!pl}$ (respectively $\Cal{T}^{E}$) the linear space spanned by the elements of $T^{E}_{\!pl}$ (respectively $T^{E}$). Any rooted tree $\sigma$ with branches $\sigma_1, \ldots, \sigma_k$ and a root $\racineLabaa$, can be written as:
\begin{equation}\label{B+}
\sigma = B_{+,\,a}(\sigma_1 \cdots \sigma_k),
\end{equation}
where $B_{+,\,a}$ is the operation which grafts a monomial $\sigma_1\, \cdots\,\sigma_k$ of rooted trees on a common root decorated by an element $a$ in $E$, which gives a new rooted tree. The planar rooted tree $\sigma$ in formula \eqref{B+} depends on the order of the branches, whereas this order is not important for the corresponding (non-planar) tree.\\ 

Define the (left) Butcher product $\,\lbutcher$ of any planar rooted trees $\sigma$ and $\tau$ by:
\begin{equation}\label{butcher product}
 \sigma \lbutcher \tau := B_{+,\,a}(\sigma \tau_1 \cdots \tau_{k}),
\end{equation}
where $\tau_1, \ldots, \tau_{k}\!\in\!T^{E}_{\!\!pl}\!,$ such that $\tau = B_{+,\,a}(\tau_1 \cdots \tau_{k})$\label{E4}. It maps the pair of trees $(\sigma, \tau)$ into a new planar rooted tree induced by grafting the root of $\sigma$, on the left via a new edge, on the root of $\tau$.\\

The left grafting $\searrow$ is a bilinear operation defined on the vector space $\Cal{T}^{E}_{\!\!\!pl}$, such that for any planar rooted trees $\sigma$ and $\tau$:
\begin{equation}\label{searrow}
\sigma \searrow \tau = \sum_{v\,vertex\,of\,\tau}{\sigma \searrow_{v}\tau},
\end{equation}
where $\sigma \searrow_{v} \tau$ is the tree obtained by grafting the tree $\sigma$, on the left, on the vertex $v$ of the tree $\tau$, such that $\sigma$ becomes the leftmost branch, starting from $v$, of this new tree. For example:
$$ \arbrea \lbutcher \arbreba = \arbredf,\,\,\,\,\,\,\,\,\,\,\arbrea \searrow \arbreba = \arbreda+\arbredc+\arbredf.$$
The number of trees in $T^{E}_{\!pl}$\label{plE2} is the same than in $T^{bin, E}_{\!pl}$: a one-to-one correspondence between them is given by D. Knuth's rotation correspondence \cite{D.K68} (see subsection 2.1 in \cite{AM2014}). On the other hand, for any homogeneous component $T^n$ of (non-planar) undecorated rooted trees of degree "$n$", for $n \geq 1$, the number  of trees in $T^n$ is given by the sequence: $1, 1, 2, 4, 9, 20, 48, \ldots,$ which is sequence A000081 in \cite{S}.\\

The two graftings, defined by \eqref{butcher product} and \eqref{searrow} above, provide the space $\Cal{T}^{E}_{\!\!\!pl}$ with structures of free magmatic algebras. K. Ebrahimi-Fard and D. Manchon showed, in their joint work (unpublished)\footnote{ More details about this work in \cite[subsection $2.1$]{AM2014}.}, that these two structures on $\Cal{T}^{E}_{\!\!\!pl}$ are isomorphic.\\

In the non-planar case, the usual product $\butcher$ given by the same formula \eqref{butcher product}, is known as the Butcher product. It is non-associative permutative (NAP), i.e. it satisfies the following identity:
$$ s \butcher (s' \butcher t) = s' \butcher ( s \butcher t),$$ 
for any (non-planar) trees $s,s',t$. The grafting product $\to$ defined as a bilinear map on the vector space $\Cal{T}^E$ as follows:
\begin{equation}\label{pre-lie product}
s \to t =\sum_{v\,\in V(t)}{ s \to_v t},
\end{equation}
for any $s, t\!\in\!\Cal{T}^E\!$, where $s \to_v t$ is the (non-planar) decorated rooted tree obtained by grafting the tree $s$ on the vertex $v$ of the tree $t$. This product is pre-Lie (see Paragraph \ref{pre-Lie algebras}). For the case with one generator, we have:
$$\arbrea \butcher \arbrea= \arbrecc,\hskip 8mm \arbrea\to\arbrea=\arbreca+\arbrecc.$$

\section{Lie and pre-Lie algebras }\label{section-two}
A Lie algebra over a field $K$ is a $K$-vector space $\Cal{L}$\,, with a $K$-bilinear mapping  $[\cdot\,, \cdot]:\Cal{L} \times \Cal{L} \rightarrow \Cal{L}$ (the Lie bracket), satisfying the following properties:
\begin{equation}\label{reflexif property}
[x, y]+[y,x] = 0 \hbox{ (anti-symmetry) }
\end{equation}
\begin{equation}\label{Jacobi identity}
[[x, y], z] + [[y, z], x] + [[z, x], y] = 0\hbox{ (Jacobi identity) }
\end{equation}
for all $x, y, z \in \Cal{L}$.\\

A left pre-Lie algebra is a vector space $\Cal{A}$ over a field $K$\label{not5-K}, together with a bilinear operation ''$\vartriangleright$'' that satisfies:
\begin{equation}\label{pre-lie identity}
(x \vartriangleright y)\vartriangleright z - x \vartriangleright  (y \vartriangleright z) = (y \vartriangleright x)\vartriangleright z - y \vartriangleright  (x \vartriangleright z), \forall x,y,z \in \Cal{A}.
\end{equation}   
The identity \eqref{pre-lie identity} is called the left pre-Lie identity, and it can be written as:
\begin{equation}
L_{[x,y]} = [L_x, L_y], \forall x,y \in \Cal{A},
\end{equation}      
where for every element $x$ in $\Cal{A}$, the linear transformation $L_x$ of the vector space $\Cal{A}$ is defined by $L_x(y)={x}\vartriangleright{y}, \forall y\!\in\!\Cal{A}$, and $[x,y]={x}\vartriangleright{y}-{y}\vartriangleright{x}$ is the commutator of the elements $x$ and $y$ in $\Cal{A}$. The usual commutator $[L_x,L_y]=L_xL_y-L_yL_x$ of the linear transformations of $\Cal{A}$ defines a structure of Lie algebra over $K$ on the vector space $L(\Cal{A})$ of all linear transformations of $\Cal{A}$. For any pre-Lie algebra $\Cal{A}$, the bracket $[\cdot,\cdot]$ satisfies the Jacobi identity, hence induces a structure of Lie algebra on $\Cal{A}$.\\

\subsection{Free Lie algebras} The Lie algebra of Lie polynomials, introduced by E. Witt (1911-1991), is actually the free Lie algebra. The first appearance of Lie polynomials was at the turn of the century in the work of Campbell, Baker and Hausdorff on the exponential mapping in a Lie group, when the well-known result "Campbell-Baker-Hausdorff formula" appeared. For more details about a historical review of  free Lie algebras, we refer the reader to the reference \cite{Ch.R93} and the references therein.  \\

A free Lie algebra is a pair $\big(\Cal{L}, i\big)$, of a Lie algebra $\Cal{L}$ together with a map $i: E \rightarrow \Cal{L}$ from a (non-empty) set $E$ into $\Cal{L}$, satisfying the following universal property: for any Lie algebra $\Cal{L}'$ and any mapping $f:E \rightarrow \Cal{L}'$, there is a unique Lie algebra homomorphism $\tilde{f}: \Cal{L} \rightarrow \Cal{L}'$ which makes the following diagram commute:
\begin{figure}[h]
\diagrama{
\xymatrix{
E \ar[r]^{i} \ar[dr]_{f} & \Cal{L} \ar[d]^{\tilde{f}}\\
 & \Cal{L}'}}
\caption{The universal property of the free Lie algebra.}
\label{up}
\end{figure}

\noindent It is unique up to an isomorphism. If $\Cal{L}$ is a $K$- Lie algebra and $E \subseteq \Cal{L}$, then we say that $E$ freely generates $\Cal{L}$ if $\big(\Cal{L}, i\big)$ is free, where $i$ is the canonical injection from $E$ to $\Cal{L}$.\\  

Recall \cite{J.D74, Ch.R93} that the enveloping algebra $\Cal{U}(\Cal{L})$ of the free Lie algebra $\Cal{L}(E)$ is the free unital associative algebra on $E$. The Lie algebra homomorphism $\varphi_{0}: \Cal{L}(E) \rightarrow \Cal{U}(\Cal{L})$ is injective, and $\varphi_{0}(\Cal{L}(E))$ is the Lie subalgebra of $\Cal{U}(\Cal{L})$ generated by $j(E)$, where $j := \varphi_{0} \circ i$.  

\subsection{Gr\"obner bases} The theory of Gr\"obner bases was introduced in 1965 by Bruno Buchberger for ideals in polynomial rings and an algorithm called Buchberger algorithm for their computation. This theory contributed, since the end of the Seventies, in the development of computational techniques for the symbolic solution of polynomial systems of equations and in the development of effective methods in Algebraic Geometry and Commutative Algebra. Moreover, this theory has been generalized to free non-commutative algebras and to various non-commutative algebras of interest in Differential Algebra, e.g. Weyl algebras, enveloping algebras of Lie algebras \cite{T.M}, and so on.  \\

The attempt to imitate Gr\"obner basis theory for non-commutative algebras works fine up to the point where the termination of the analogue to the Buchberger algorithm can be proved. Gr\"obner bases and Buchberger algorithm have been extended, for the first time, to ideals in free non-commutative algebras by G. Bergman in 1978. Later, F. Mora in 1986 made precise in which sense Gr\"obner bases can be computed in free non-commutative algebras \cite{T.M}. The construction of finite Gr\"obner bases for arbitrary finitely generated ideals in non-commutative rings is possible in the class of \textit{solvable algebras} \footnote{For more details about the solvable algebras see \cite[Appendix: Non-Commutative Gr\"obner Bases, pages 526-528]{TB93}.}. This class comprises many algebras arising in mathematical physics such as: Weyl algebras, enveloping algebras of finite-dimensional Lie algebras, and iterated skew polynomial rings. Gr\"obner bases were studied, in these algebras, for special cases by Apel and Lassner in 1985, and in full generality by Kandri-Rody and Weispfenning in 1990 \cite{TB93}.  \\

Recently, V. Drensky and R. Holtkamp used Gr\"obner theory  in their work \cite{D.H.} for a non-associative, non-commutative case (the magmatic case). Whereas, L. A. Bokut, Yuqun Chen and Yu Li, in their work \cite{BYL}, give Gr\"obner-Shirshov basis for a right-symmetric algebra (pre-Lie algebra). The theory of Gr\"obner-Shirshov bases was invented by A. I. Shirshov for Lie algebras in 1962 \cite{AS62}.\\

We try in this paper, precisely in Section \ref{mwo}, to describe a monomial basis in tree version for the free Lie (respectively pre-Lie) algebras using the procedures of Gr\"obner bases, comparing with the one (i.e. the monomial basis) obtained for the free pre-Lie algebra in our preceding work \cite{AM2014}.  We need here to review some basics for the theory of Gr\"obner bases.     

\begin{defn}\label{monomial order}
 Let $\big( M(E), \cdot\big)$\label{E6} be the free magma generated by $E$. A total order $<$ on $M(E)$ is said to be monomial if it satisfies the following property:
\begin{equation}\label{property}
\hbox{ for any } x, y, z \in M(E), \hbox{ if } x < y, \hbox{ then } x\cdot z < y \cdot z \hbox{ and } z \cdot x < z \cdot y,
\end{equation} 
i.e. it is compatible with the product in $M(E)$.
\end{defn}

This property, in \eqref{property}, implies that for any $x, y \in M(E)$ then $x < x \cdot y$. An order is called  a well-ordering if every strictly decreasing sequence of monomials is finite, or equivalently if every non-empty set of monomials has a minimal element.\\

Let $\Cal{M}_E$ be the $K$\label{not4-K}-linear span of the free magma $M(E)$, and $I$ be any magmatic (two-sided) ideal of $\Cal{M}_E$.  For any element $f = \!\!\!\!\!\!\sum\limits_{x \in M(E)}^{}{\!\!\!\lambda_x x}$ (finite sum) in $I$, define $T(f)$\label{Tf} to be the maximal term of $f$ with respect to a given monomial order defined on $M(E)$, namely $T(f) = \lambda_{x_0} x_0, \hbox{ with } x_0 =max \{ x \in M(E), \lambda_x \neq 0\}$.  Denote $T(I) := \{ T(f) : f \in I \}$ the set of all maximal terms of elements of $I$. Note that the set $T(I)$ forms a (two-sided) ideal of the magma $M(E)$ \cite{T.M}. Define the set $O(I):= M(E) \backslash T(I)$. We have that the magma $M(E)= T(I)\cup O(I)$ is the disjoint union of $T(I), O(I)$ respectively. As a consequence, we get that:

\begin{equation}\label{realized2}
\Cal{M}_E =S\!pan_K(T(I))\oplus S\!pan_K(O(I)).
\end{equation}

Define a linear mapping $\varphi$ from $I$ into $S\!pan_K(T(I))$, which makes the following diagram commute:
\begin{figure}[h]
\diagrama{
\xymatrix{
I \ar@{^{(}->}[r]^-{i} \ar@{->}[drr]^-{\varphi}&\Cal{M}_E \ar@{->}[r]^-{\tilde{=}} & S\!pan_K(T(I))\oplus S\!pan_K(O(I)) \ar@{->>}[d]^{P}\\
& & S\!pan_K(T(I))}}
\caption{Definition of $\varphi.$}
\label{}
\end{figure}\\

\noindent where $P$ is the projection map. Then the mapping $\varphi$ is defined by:
\begin{equation}\label{def-phi}
\varphi(f) = \sum\limits_{x \in T(I)}{\alpha_x x}, \hbox{ for } f \in I,
\end{equation} 
where $f= \sum\limits_{x \in T(I)}{\alpha_x x} + \hbox{ corrective term in } S\!pan_K(O(I))$, and $\alpha_x \in K \hbox{ for all } x \in T(I)$. The map $\varphi$ is obviously injective. Indeed, for any $f \in I$ and $\varphi(f) = 0$, then $f \in S\!pan_K(O(I))$, and from Theorem \ref{span}, $S\!pan_K(O(I)) \bigcap I = \{ 0 \}$. Also, according to Theorem \ref{span} and by the definition of $\varphi$ in \eqref{def-phi}, we note that $\varphi$ is surjective. Hence, $\varphi$ is an isomorphism of vector spaces. Thus, we can deduce from the formula \eqref{realized2}:

\begin{equation}
\Cal{M}_E =I \oplus S\!pan_K(O(I)).
\end{equation}
In Section \ref{mwo}, we will give a tree version of the monomial well-ordering with a review of Mora's work \cite{T.M}, in the case of rooted trees. 
\subsection{Free pre-Lie algebras}\label{pre-Lie algebras}
As a particular example of pre-Lie algebras, take the linear space of the set of all (non-planar) $E$-decorated rooted trees $\Cal{T}^E$\label{cnp3} which has a structure of pre-Lie algebra together with the product "$\to$" defined in \eqref{pre-lie product}.\\

Free pre-Lie algebras have been handled in terms of rooted trees by F. Chapoton and M. Livernet \cite{CL01}, who also described the pre-Lie operad explicitly, and by A. Dzhumadil'daev and C. L\"ofwall independently \cite{AC02}. For an elementary version of the approach by Chapoton and Livernet without introducing operads, see e.g. \cite[Paragraph 6.2]{D.M}:

\begin{thm}\label{d-generators}
Let $k$ be a positive integer. The free pre-Lie algebra with $k$ generators is the vector space $\Cal T$ of (non-planar) rooted trees with $k$ colors, endowed with grafting.
\end{thm}

\section{A pre-Lie structure on free Lie algebras}\label{CIIIsecI}
Let $\Cal{L}(E)$ be the free Lie algebra generated by a (non-empty) disjoint union of subsets $E = \bigsqcup\limits_{i\in N} E_{i}$, where $E_i$ is the subset of elements $a^i_1, \ldots, a^i_{d_i}$ of degree $i$, and $\#E_i=d_i$. The free Lie algebra $\Cal{L}(E)$ can be graded, using the grading of $E$:
\begin{equation}
\Cal{L}(E) = \bigoplus_{i \in N} {\Cal{L}_i}, 
\end{equation}
where $\Cal{L}_i $ is the subspace of all elements of $\Cal{L}(E)$ of degree $i$. In particular $E_i \subset \Cal{L}_i$. Define an operation $\rhd$ on $\Cal{L}(E)$ by:
\begin{equation}\label{rhd}
x \rhd y := \frac{1}{|x|} [x,y], 
\end{equation} 
for $x, y \in \Cal{L}(E)$.
\begin{prop}\label{rhd is pre-lie}
 The operation $\rhd$ defined by \eqref{rhd} is a bilinear product which satisfies the pre-Lie identity. 
\end{prop}
\begin{proof}
For $x, y, z \in \Cal{L}(E)$, we have: 
\begin{align*}
(x\rhd y) \rhd z - x \rhd (y \rhd z) &= \frac{1}{|x|} [x,y] \rhd z - \frac{1}{|y|} x \rhd [y,z] &  \\ 
&= \frac{1}{|x| \big(|x| + |y|\big)} [[x,y],z] - \frac{1}{|x||y|} [x,[y,z]]& \\
&= \frac{1}{|x| \big( |x| + |y| \big)} [[x,y],z] - \frac{1}{|x||y|} \big([[x,y],z] - [y,[z,x]]\big),\hbox{ since }&  \\ 
& \; \; \; \; [x,[y,z]] + [z,[x,y]] + [y,[z,x]] = 0 \hbox{ (the Jacobi identity) } & \\
&= \left(\frac{1}{|x|}\right) \left(\frac{|y| - \big(|x| + |y|\big)}{|y| \big(|x| + |y|\big)} \right) [[x,y],z] + \frac{1}{|x||y|} [y,[z,x]] & \\
&=\frac{1}{|y| \big(|x| + |y|\big)} [[y,x],z] - \frac{1}{|x||y|} [y,[x,z]] & \\ 
&= (y\rhd x) \rhd z - y \rhd (x \rhd z).&
\end{align*}
Then $\Cal{L}(E)$\label{CL3} together with $\rhd$ forms a graded pre-Lie algebra generated by $E$.
\end{proof}

A pre-Lie algebra structure can be put this way on any $\mathbb{N}$-graded Lie algebra $\Cal{L}$ such that $\Cal{L}_0 = \{0\}$. Another pre-Lie bracket, proposed on $\Cal{L}$ by T. Schedler \cite{T.S}\footnote{For more details about this construction of pre-Lie algebra see \cite[Proposition 3.3.3]{T.S} and \cite{LF}.}, is given by:

\begin{equation}\label{def. of Brhd}
x \blacktriangleright y = \frac{|y|}{|x| + |y|} [x, y], \hbox{ for any } x, y \in \Cal{L}.
\end{equation}

\noindent These two constructions are isomorphic, via the linear map: 

\[ 
\alpha : \left\{
                \begin{array}{ll}\label{def of alpha}
                  \big(\Cal{L}, \blacktriangleright\!\!\big)  \longrightarrow \big(\Cal{L}, \rhd\big),\\
                  \;\;\;\;\;\;\,x \longmapsto |x| x.
                \end{array}
              \right.
							\]
Indeed, $\alpha$ is a bijection, and for any $x, y \in \Cal{L}$ we have:
\begin{align*}
\alpha( x \blacktriangleright y) & = \alpha(\frac{|y|}{|x| + |y|} [x, y]), \hbox{ (by the definition of $\blacktriangleright$ in \eqref{def. of Brhd}), }&\\
& =  \frac{|y|}{|x| + |y|} \alpha([x, y])&\\
& =  \frac{|y|}{|x| + |y|} (|x| + |y|) [x, y], \hbox{ (by the definition of $\alpha$ above), }&\\
& = |y| [x, y]&\\
& = \frac{|x| |y|}{|x|} [x, y]&\\
& = (|x| |y|) x \rhd y, \hbox{ (by the definition of $\rhd$ in \eqref{rhd}), }&\\
& = |x| x \rhd |y| y&\\
& = \alpha(x) \rhd \alpha(y).&
\end{align*} 						
						
Denote by $[\cdot, \cdot]_ {\rhd}$\label{CL4} the underlying Lie bracket induced by the pre-Lie product $\rhd$, which is defined by:
\begin{equation}\label{nbr}
[x, y]_ {\rhd} = x \rhd y - y \rhd x, \hbox{ for } x, y \in \Cal{L}.
\end{equation}
Then the two Lie structures defined on $\Cal{L}$ by the Lie brackets $[\cdot, \cdot]$, $[\cdot, \cdot]_ {\rhd}$ respectively,   are also isomorphic via $\alpha$. Indeed, by substituting the pre-Lie product $\rhd$, described in \eqref{rhd}, by the Lie bracket $[\cdot, \cdot]$ in the definition of the Lie bracket $[\cdot, \cdot]_ {\rhd}$ in \eqref{nbr}, we get:
\begin{equation}\label{intermid}
[x, y]_{\rhd}  = \frac{1}{|x|} [x, y] - \frac{1}{|y|} [y,x] = \frac{|x| + |y|}{|x|\,|y|} [x, y], \hbox{ for any } x, y \in \Cal{L},
\end{equation}
but, 
\begin{align*}
\alpha\big([x, y]\big) = |[x, y]|\, [x, y] &=\big(|x| + |y|\big) [x, y] &\\
&= |x| |y| [x, y]_{\rhd} &  \hbox{ (by \eqref{intermid}) } \\
&= [|x| x, |y| y]_ {\rhd}& \\
&= [\alpha(x), \alpha(y)]_ {\rhd}& \hbox{ (by \eqref{def of alpha}). }
\end{align*} 

For any (non-planar) rooted tree $t$, we can decorate the vertices of $t$ by elements of $E$, by means of a map $d: V(t) \rightarrow E$, where $V(t)$ is the set of vertices of $t$. Denote by $T^{E}$ the set of all (non-planar) rooted trees decorated by the elements of $E$, define the degree $|t|$ of a decorated tree $t$ in $T^{E}$ by:
\begin{equation}\label{degree of tree}
|t|:= \sum_{v \in V(t)} {|d(v)|} 
\end{equation}

In particular, there is a unique pre-Lie homomorphism $\Phi$ from $\big(\Cal{T}^{E}, \to\!\!\big)$ onto $\big(\Cal{L}(E), \rhd\big)$\label{CL5}, such that:
\begin{equation}\label{hom}
\Phi(\racineLabaa)= a \hbox{ for any } a \in E. 
\end{equation}
If we take $t= t_1 \to (t_2 \to(\cdots \to(t_k \to \racineLabaa) \cdots ))  \in \Cal{T}^{E}\!\!,$ then:
\begin{equation}\label{homom}
\Phi(t)= x_1 \rhd (x_2 \rhd ( \cdots \rhd (x_k \rhd a) \cdots )), 
\end{equation}
with $x_i = \Phi(t_i) , \hbox{ and } |t_i|= |x_i|, \forall i=1, \ldots, k$\label{deg1}. Let $I$ be the two-sided ideal of $\Cal{T}^{E}$ generated by all elements on the form:
\begin{equation}\label{def of I}
|s|\big(s \to t\big) + |t|\big(t \to s\big), \hbox{ for } s, t \in \Cal{T}^{E}.
\end{equation}
The ideal $I$ satisfies the following properties:

\begin{prop}\label{quotient}
The quotient $\Cal{L}'(E):=\Cal{T}^{E}/ I$ has structures of pre-Lie algebra and Lie algebra, respectively. 
\end{prop}

\begin{proof}
Using the pre-Lie grafting $\to$ defined on $\Cal{T}^{E}$\!, we can define the following operations on $\Cal{L}'(E)$:
\begin{equation}\label{pre-lie}
\overline{s} \rhd^{\!\!*} \overline{t} := \overline{s} \to \overline{t}:= \overline{s \to t}, 
\end{equation}
\begin{equation}\label{lie bracket}
[\overline{s}, \overline{t}] := \overline{[s, t]}:= |s| \overline{s \to t}, 
\end{equation} 
for any $s, t \in \Cal{T}^{E}$, where the bar stands for the class modulo $I$. The product in \eqref{pre-lie} is pre-Lie by definition. The bracket defined in \eqref{lie bracket} is well-defined and satisfies the following identities:
\begin{enumerate}
\item The anti-symmetry identity: for any $s, t \in \Cal{T}^{E}$, we have
\begin{align*}
[\overline{s}, \overline{t}] = - [\overline{t}, \overline{s}], \hbox{ since, } |s|(s \to t) + |t|(t \to s) \in I. 
\end{align*}

\item\label{J.I} The Jacobi identity: for any $s, t, t' \in \Cal{T}^{E}$, then
\begin{align*}
[\overline{s}, [\overline{t}, \overline{t'}]] + [[\overline{s}, \overline{t'}], \overline{t}] &= |s| |t|\,\big( \overline{s \to (t \to t')}\big) &\\ & + |s|(|s|+|t'|)\, \big(\overline{(s \to t') \to t} )\big) & \\
\hbox{ (using the anti-symmetry identity) } \longrightarrow &= |s| |t|\, \big( (\overline{s \to (t \to t')})  -  (\overline{t \to (s \to t')}) \big) 
& \\
\hbox{ (using the pre-Lie identity) } \longrightarrow &= |s| |t|\,\big( \big( \overline{(s \to t) \to t'} \big) - \big( \overline{(t \to s) \to t'} \big) \big)& \\
&=|s| |t|\,\big( (\overline{s \to t - t \to s} ) \to \overline{t'} \big) \\
\hbox{ (using the anti-symmetry identity) } \longrightarrow &=|s| |t| \big(( \overline{s \to t} + \frac{|s|}{|t|} \overline{s \to t})\to \overline{t'}\big)& \\
&=|s| |t| \frac{|s|+|t|}{|t|}\, \big( (\overline{s \to t}) \to \overline{t'}\big)& \\
&=|s| (|s|+|t|)\,\big( \overline{(s \to t) \to t'}\big)&\\
&=[[\overline{s}, \overline{t}], \overline{t'}].& 
\end{align*}
\end{enumerate}  
\end{proof}
\begin{prop}\label{hom2}
$I = Ker\,\Phi.$ 
\end{prop}

Let $\big(M(E), \cdot \big)$ be the free magma generated by $E$\label{E26}, and let $\Cal{M}_E$ be the free magmatic algebra generated by $E$, i.e. the linear span of the magma $M(E)$. Define a new magmatic product $\ast$ on $M(E)$ by:
\begin{equation}
x \ast y := |x| x \cdot y
\end{equation}
for any $x, y \in M(E)$, and extend bilinearly. We need, to prove Proposition \ref{hom2}, to introduce the following lemmas.

\begin{lem}\label{magmatic isomorphism}
The two magmatic algebras $\big(\Cal{M}_E, \cdot\big)$ and $\big(\Cal{M}_E, \ast\big)$ are isomorphic. 
\end{lem}

\begin{proof}
By universal property of the free magmatic algebra, there is a unique morphism $\gamma : \big(\Cal{M}_E, \cdot\big) \rightarrow \big(\Cal{M}_E, \ast\big)$ such that $\gamma(a) = a$, for any $a \in E$. For any $x, y \in \Cal{M}_E$, we have:
\begin{equation}
\gamma(x \cdot y) = \gamma(x) \ast \gamma(y) = |\gamma(x)|\,\gamma(x) \cdot \gamma(y). 
\end{equation}
Hence one can see, by induction on the degree of elements of the magma $M(E)$, that we have for any $z \in M(E)$:
\begin{equation}\label{definition of magmma}
\gamma(z) = f(z)\, z,
\end{equation}
where $f: M(E) \rightarrow \mathbb{N}$ is recursively given by:
$f(a) =1$, for any $a \in E$, and $f(x \cdot y) = |x| f(x) f(y)$ for $x, y \in M(E)$ (for more details about this mapping see Example \ref{details} below). Hence $\gamma$ is an isomorphism.  
\end{proof}
Now, let $J$ be the two-sided ideal generated by the the anti-symmetry and the Jacobi identities on $\big( \Cal{M}_E, \ast \big)$, and let $J'$ be the two-sided ideal of $\big(\Cal{M}_E, \cdot \big)$ generated by the pre-Lie identity and the elements on the form:
\begin{equation}\label{weighted anti-symmetry} 
|x| x \cdot y + |y| y \cdot x, \hbox{ for x, y } \in M(E).
\end{equation}

\begin{lem}\label{relation between J and J'}
$J = J'$.
\end{lem}

\begin{proof}
Let $J'_1$ be the ideal generated by the elements \eqref{weighted anti-symmetry}. Equivalently, $J'_1$ is generated by the elements $x\ast y + y \ast x, \hbox{ for } x, y \in M(E)$. We have:
\begin{align*}
x \cdot(y \cdot z) - (x \cdot y) \cdot z - y \cdot (x \cdot z) + (y \cdot x) \cdot z = &\frac{1}{|x| |y|} x \ast (y \ast z) - \frac{1}{|x|(|x|+|y|)}(x \ast y) \ast z&\\
   &- \frac{1}{|x| |y|}y \ast (x \ast z) +  \frac{1}{|y|(|x|+|y|)}(y \ast x) \ast z&\\
= &\frac{1}{|x||y|(|x|+|y|)} \big( (|x|+|y|) x \ast (y \ast z) - |y| (x \ast y) \ast z&\\
   &- (|x|+|y|) y \ast (x \ast z) +  |x|(y \ast x) \ast z \big)&\\
= & \frac{1}{|x| |y|} \big( - y \ast (x \ast z) + (y \ast x) \ast z + x \ast(y \ast z) \big)&\\
 &- \frac{1}{|x|(|x|+|y|)} ( y \ast x + x \ast y) \ast z&\\
= & \frac{1}{|x| |y|} \big(x \ast (y \ast z) + y \ast ( z \ast x) + z \ast(x \ast y) \big)\,modulo\,J'_1,&\\
\end{align*}
hence $x \cdot(y \cdot z) - (x \cdot y) \cdot z - y \cdot (x \cdot z) + (y \cdot x) \cdot z \in J$. This means $J' \subset J$.\\

Conversely,
\begin{equation}\label{conversely}
x \ast (y \ast z) + y \ast ( z \ast x) + z \ast(x \ast y) = |x||y| \big(x \cdot(y \cdot z) - (x \cdot y) \cdot z - y \cdot (x \cdot z) + (y \cdot x) \cdot z \big)\,modulo\,J'_1,
\end{equation}
hence the left-hand side of \eqref{conversely} belongs to $J'$, which proves the inverse inclusion.
\end{proof}

\begin{proof}[Proof of Proposition \ref{hom2}]
 The free pre-Lie algebra generated by $E$ is given by $\Cal{T}^E$ \cite{CL01}, \cite{AC02}. Hence, the quotient $\Cal{L}'(E)= \big(\Cal{M}_E, \cdot\big) / J' = \Cal{T}^E / I$ is a pre-Lie (respectively Lie) algebra. The Lie algebra $\Cal{L}(E) = \big(\Cal{M}_E, \ast\big) / J$ carries a pre-Lie algebra structure induced by the product defined in \eqref{rhd}, such that the free pre-Lie algebra $\Cal{P\!L}(E):= \Cal{M}_E / J'_2 = \Cal{T}^E$, where $J'_2$ is the two-sided ideal generated by the pre-Lie identity on $\big( \Cal{M}_E, \cdot \big) $\label{E27}, is homomorphic to $\Cal{L}(E)$ by $\Phi$ described in \eqref{hom} and \eqref{homom}, as pre-lie algebras, as in the commutative diagram in Figure \ref{description of relations}, where $q, q'$ therein are quotient maps. 
\ignore{
\begin{figure}[h]
\diagrama{
\xymatrix{
E \ar@{^{(}->}[r]^-{i}\ar@{_{(}->}[dr]_{j} & \big( \Cal{M}_E, \cdot \big) \ar@{->>}[dr]^{q'} \ar@{-->}[d]^{Id} \ar@{->>}[rrr]^{q}& & &L'(E) \ar@{->}[ddll]^{\widetilde{\gamma}}\\
& \big( \Cal{M}_E, \ast \big) \ar@{->>}[dr]^{q'''} &  \Cal{T}^{E}  \ar@{->>}[urr]^{q''} \ar@{->>}[d]^{\Phi}\\
& & \Cal{L}(E)& &}}
\caption{}
\label{description of relations}
\end{figure}
}
\begin{figure}[!]
\diagrama{
\xymatrix{
E \ar@{^{(}->}[r]^-{i}\ar@{_{(}->}[dr]_{j} & \big( \Cal{M}_E, \cdot \big) \ar@{->>}[dr]^{q'} \ar@{-->}[d]^{Id} & \\
& \big( \Cal{M}_E, \ast \big) \ar@{->>}[dr]^{q} &  \Cal{T}^{E}   \ar@{->>}[d]^{\Phi}\\
& & \Cal{L}(E)}}
\caption{}
\label{description of relations}
\end{figure}
From Figure \ref{description of relations} and Lemmas \ref{magmatic isomorphism}, \ref{relation between J and J'}, we get that:
\begin{align*}
&\mop{Ker}\big(\Phi \circ q'\big) = J' = J=\mop{Ker}q,\hbox{ and then } Ker\,\Phi = q'(J') = q'(J) = I,&
\end{align*}
therefore Proposition \ref{hom2} is proved.
\end{proof}
Note that the Lie product on $\Cal{L}(E)$ is the image of $\ast$ by $\Phi \circ q'$. The pre-Lie product $\rhd$ is the image of "$\cdot$" by $\Phi \circ q'$. Hence, we recover Proposition \ref{rhd is pre-lie} this way. 
\begin{exam}\label{details}
The free magma $M(E)$ can also be identified with the set of all planar binary rooted trees, with leaves decorated by the elements of $E$, together with the product $\vee$ \label{V1}defined in section \ref{section-one}. For instance,

\begin{equation}
a \cdot b = \treeLab,\,\, (a \cdot b) \cdot c = \treeLaba, \,\, a \cdot (b \cdot c) = \treeLabb, \,\,\hbox{ and } z = x \cdot y = \big((a \cdot b) \cdot c\big) \cdot ( d \cdot e ) = \treeLabc\,\,, 
\end{equation}
with $x:= (a \cdot b) \cdot c$, and $y:= d \cdot e$. Then:
\begin{align*}
f(z) &= f(x \cdot y)&\\
&= |x| f(x) f(y)&\\
&= |x| \big(|a \cdot b| f(a \cdot b)f(c)\big)\,\big(|d| f(d)f(e)\big)&\\
&= |x| \big( (|a| + |b|)\,\big(|d| \big(|a| f(a) f(b) f(c) f(d) f(e) \big)\big) \big)&\\
&= |a|\,|d|\,(|a| + |b|)\,(|a| + |b| + |c|)\;\;\;\;\;\big(\hbox{ since, } f(a) f(b) f(c) f(d) f(e) = 1\big).&   
\end{align*}
There is another description of $f$, detailed as follows: in a planar binary tree, there are two types of edges, going on the left (from bottom to top) or going on the right. Consequently, except the root, there are two types of vertices, the left ones (the incoming edge on the left) and the right ones. Let $t$ be a planar binary tree, with leaves decorated by elements of $E$, then $f(t)$ is the product over all left vertices $v$ of the sums of the degree of the decorations of the leaves $l$ with a path from $v$ to $l$.  
\end{exam}

\noindent Consequently, from Propositions \ref{rhd is pre-lie}, \ref{quotient} and \ref{hom2}, we get the following result.
\begin{cor}
There is a unique pre-Lie (respectively Lie) isomorphism between $\Cal{L}'(E)$ and $\Cal{L}(E)$, such that $\Phi(a\,mod.J') = a\,mod.J$, for any $a \in E$\label{E28}. 
\end{cor}

\section{A monomial well-order on the planar rooted trees, and applications}\label{mwo}
Let $E$ be a disjoint union $E:=\bigsqcup\limits_{n \geq 1} E_n$ of finite subsets $E_n=\{ a^n_1, \ldots, a^n_{d_n}\}$, where $E_n$ is the subset of all elements of $E$ of degree $n$. Let us order the elements of $E$ by:
\begin{equation}\label{order of generators}
a^1_1 < \cdots < a^1_{d_1} < a^2_1 < \cdots < a^2_{d_2} < \cdots < a^i_1 < \cdots < a^i_{d_i} < \cdots  
\end{equation}
Some particular sets $E$ of generators can be considered:
\begin{enumerate}
\item $E=\bigsqcup\limits_{n \geq 1}E_n,\hbox{ where } \#E_i=0 \hbox{ or }1$. A particular situation is:
\begin{enumerate}
	\item\label{E finite} take $E=\{a_1, \ldots, a_s\}, \hbox{ with } a_i \in E_i, \hbox{ and } |a_i|=i, \hbox{ for } i=1, \ldots, s.$
\end{enumerate} 
	\item $E=E_1,\hbox{ where } \#E_1=d_1=d, \hbox{ as a special case:}$
	\begin{enumerate}
	\item\label{d2} take $d_1=d=2.$
\end{enumerate}
\end{enumerate}

 The set $T^{E}_{pl}$\label{E29} forms the free magma generated by the set $\{\racineLabaa : \hbox{ for } a \in E\}$, under the left Butcher product $\lbutcher$. Define a total order $\preceq$ on $T^{E}_{pl}$ as follows: 

\begin{equation}\label{order}
\hbox{ for any $\sigma, \tau \in T^{E}_{pl}$, then} \,\sigma \preceq \tau \hbox{ if and only if } 
\end{equation}
\begin{enumerate}
\item $|\sigma| < |\tau|, \hbox{ or }$:
\item $|\sigma| = |\tau| \hbox{ and } b(\sigma) <  b(\tau),$ or:
\item $|\sigma| = |\tau|, b(\sigma) =  b(\tau)$ and $(\sigma_1, \ldots, \sigma_k) \preceq (\tau_1, \ldots, \tau_k)$ lexicographically, where $\sigma= B_{+, r}(\sigma_1 \ldots \sigma_k),$ $ \tau=B_{+, r'}(\tau_1 \ldots \tau_k)$ , or:
\item  $|\sigma| = |\tau|, b(\sigma) =  b(\tau)$, $(\sigma_1, \ldots, \sigma_k) = (\tau_1, \ldots, \tau_k)$ and the root $r$ of $\sigma$ is strictly  smaller than the root $r'$ of $\tau$.
\end{enumerate} 
where $k=b(\sigma)$ is the number of branches of $\sigma$ starting from the root. This order depends on an ordering of the generators, here we order them by:
\begin{equation}
\racineLabell \prec \cdots \prec \racineLabeldl \prec \cdots \prec \racineLabeli \prec \cdots \prec \racineLabeidi \prec \cdots
\end{equation}
 like in \eqref{order of generators}. The first terms in $T^{E}_{pl}$, when $E=\{a^1, a^2\}$, are ordered by  $"\prec"$ as follows:
\begin{align*}
&\racineLab1\prec\racineLabb\prec\arbreaLaba\prec\arbreaLabb\prec\arbreaLabc\prec\arbrebaLaba\prec\arbrebbLaba\prec\arbreaLabd\prec\arbrebaLabb\prec\arbrebaLabc\prec\arbrebaLabd\prec\arbrebbLabb\prec\arbrebbLabc\prec\arbrebbLabd\prec\arbrebaLabe\prec\arbrebaLabf\prec\arbrebaLabg\prec\arbrebbLabae\prec\arbrebbLabf\prec&\\
&\;\;\;\;\arbrebbLabg\prec\arbrebaLabh\prec\arbrebbLabh\prec\cdots,&
\end{align*}
 where $\racineLabi$ is a shorthand notation for $\racineLabei$.
\begin{prop}\label{well order}
The order $\preceq$ defined in \eqref{order} is a monomial well-order.
\end{prop}
\begin{proof}
Let $\sigma, \sigma' \in T^ E_{pl}$, such that $\sigma \preceq \sigma'$. For any $\tau \in T^ E_{pl}$, we have:
$|\tau \lbutcher \sigma| < |\tau \lbutcher \sigma'|$\label{lb3},  if $|\sigma| < |\sigma'|$, and they are equal when the degrees of  $\sigma $ and  $\sigma'$ are equal. If  $b(\sigma) < b(\sigma'),$ then $b(\tau \lbutcher \sigma) < b(\tau \lbutcher \sigma'). $ But, if  $b(\sigma) = b(\sigma') = k , $ then  $b(\tau \lbutcher \sigma) = b(\tau \lbutcher \sigma') = k+1$. Lexicographically,  $(\tau, \sigma_1, \ldots, \sigma_k) \preceq (\tau, \sigma'_1, \ldots, \sigma'_k) $ when  $(\sigma_1, \ldots, \sigma_k) \preceq (\sigma'_1, \ldots, \sigma'_k)$. The root of  $\tau \lbutcher \sigma$ is the root of  $\sigma$, the same thing for $\tau \lbutcher \sigma'$ holds. Then $\tau \lbutcher \sigma \preceq \tau \lbutcher \sigma'$. By the same way, one can verify that $\sigma \lbutcher \tau \preceq \sigma' \lbutcher \tau$. Hence, the order $\preceq$ is a monomial. Obviously, this order is a well-order.
\end{proof}

In following, we adapt the algorithm of T. Mora \cite{T.M} to find Gr\"obner bases for the free Lie algebras in tree version. For any element $f \in \Cal{T}^{E}_{pl}$, define $T(f)$\label{Tf1} to be the maximal term of $f$ with respect to the order $\preceq$ defined in \eqref{order}, and let $lc(f)$ be the coefficient of $T(f) \hbox{ in } f$, for example:\\
$$\hbox{ if } f=\arbrebaLaba+ \arbrebbLabb+2\arbrebbLabc, \hbox{ then } T(f)=\arbrebbLabc, \hbox{ and }lc(f)=2.$$ 

Let $I$ be any (two-sided) ideal of $\Cal{T}^{E}_{pl}$. Define:
\begin{equation}
T(I):=\big\{T(f) \in T^{E}_{pl}: f \in I\big\}\,, \; O(I):= T^{E}_{pl}\backslash T(I)
\end{equation}
 to be subsets of the magma $T^{E}_{pl}$, where $T(I)$\label{Tf2} forms a (two-sided) ideal of $T^{E}_{pl}$.

\begin{thm}\label{span}
If $I$ is a (two-sided) ideal of $\Cal{T}^{E}_{pl}$, then:
\begin{enumerate}
\item\label{imp} $\Cal{T}^{E}_{pl}=I\oplus S\!pan_K(O(I))$.
\item\label{isomorphic}$\Cal{T}^{E,\,*}_{pl}:= \Cal{T}^{E}_{pl} / I$ is isomorphic, as a $K$\label{not7-K}-vector space, to $S\!pan_K(O(I))$.
\item For each $f\in \Cal{T}^{E}_{pl}$ there is a unique $g:=Can(f,I) \in S\!pan_K(O(I))$, such that $f-g\in I$. Moreover:
\begin{enumerate}
\item $ Can(f,I)=Can(g,I) \hbox{ if and only if } f-g\in I$.
\item $ Can(f,I)=0 \hbox{ if and only if } f \in I$.
\end{enumerate}
\end{enumerate}
The symbol $Can(f,I)$, which satisfies the identities above, is called the canonical form of $f$ in $S\!pan_K(O(I))$.
\end{thm}
\begin{proof}
The proof is detailed in \cite[Theorem 1.1]{T.M} in the associative case. The procedure followed in the proof of \eqref{imp} consists in the following algorithm:
\small{
\begin{align*}
&f_0:=f, \phi_0:=0, h_0:=0, i:=0&\\
&\hbox{ while $f_i\neq 0$ do }&\\
&\,\,\,\,\hbox{ If $T(f_i) \,\slash\!\!\!\!\!\in\!T(I)$ then }&\\
&\,\,\,\,\,\phi_{i+1}:= \phi_i, h_{i+1}:=h_i + lc(f_i)T(f_i), f_{i+1}:= f_i - lc(f_i)T(f_i)&\\
&\,\,\,\,\hbox{ else $\%T(f_i) \in T(I) \%$ }&\\
&\,\,\,\,\,\hbox{ choose $g_i \in I$, such that $T(g_i)= T(f_i), lc(g_i)=1$ }&\\
&\,\,\,\,\,\phi_{i+1}:= \phi_i+ lc(f_i)g_i, h_{i+1}:=h_i , f_{i+1}:= f_i - lc(f_i)g_i&\\
&i:=i+1&\\
&\phi:=\phi_i, h:=h_i.&
\end{align*}
}
The correctness of this algorithm is based on the following observations:
$\forall i: \phi_i \in I, h_i \in S\!pan_K(O(I)), f_i+\phi_i+ h_i = f.$ 
Termination is guaranteed by the easy observation that if $f_n \neq 0$ then $T(f_n) <\preceq T(f_{n-1})$ and by the fact that $\preceq$ is a well-ordering.  
\end{proof}

Let $J'$ be the two-sided ideal of $\Cal{T}^{E}_{pl}$\label{cp1} generated by the pre-Lie identity and all elements on the form:
\begin{equation}\label{ideal}
 |\sigma| \sigma \lbutcher \tau + |\tau| \tau \lbutcher \sigma,\;\hbox{ for any (non-empty) trees } \sigma, \tau \in T^{E}_{pl}.  
\end{equation}

\begin{exam}
In this example we calculate $Can(f,J')$, where $f=\arbreaLabe+\arbreaLabf+\arbreaLabc+\arbrebbLabc$\label{Tf3} and $J'$ is the ideal defined by \eqref{ideal}, using the algorithm described in the proof of Theorem \ref{span} above\label{lb4}:
\begin{align*}
&f_0=\arbreaLabe+\arbreaLabf+\arbreaLabc+\arbrebbLabc, \phi_0=0, h_0=0&\\
&T(f_0)= \arbrebbLabc \in T(J'), \hbox{ choose } g_0=3 \arbrebaLabd + \arbrebbLabc \in J',\,\,lc(g_0)=1&\\
&\phi_1=3 \arbrebaLabd + \arbrebbLabc,\,\,h_1= 0, f_1=\arbreaLabe+\arbreaLabf+\arbreaLabc-3\arbrebaLabd&\\
&T(f_1)= \arbrebaLabd \in T(J'), \hbox{ choose } g_1=\frac{1}{2}\big(\arbrebaLabc+ 2\arbrebaLabd\big) = \frac{1}{2}\big(\arbreaLabb+ 2\arbreaLabc\big)\lbutcher \racineLaba \in J',\,\,lc(g_1)=1&\\
&\phi_2=3 \arbrebaLabd + \arbrebbLabc-\frac{3}{2} \arbrebaLabc - 3\arbrebaLabd,\,\,h_2= 0, f_2=\arbreaLabe+\arbreaLabf+\arbreaLabc+\frac{3}{2}\arbrebaLabc&\\
&T(f_2)= \arbrebaLabc \,\,\slash\!\!\!\!\!\in T(J'), \hbox{ then: }&\\
&\phi_3=3 \arbrebaLabd + \arbrebbLabc-\frac{3}{2} \arbrebaLabc - 3\arbrebaLabd,\,\,h_3= \frac{3}{2}\arbrebaLabc, f_3=\arbreaLabe+\arbreaLabf+\arbreaLabc&\\
&T(f_3) = \arbreaLabf \in T(J'), \hbox{ choose } g_3=\frac{1}{3}(\arbreaLabe+3\arbreaLabf) \in J'&\\
&\phi_4=3 \arbrebaLabd + \arbrebbLabc-\frac{3}{2} \arbrebaLabc - 3\arbrebaLabd + \frac{1}{3}\arbreaLabe+\arbreaLabf,\,\,h_4= \frac{3}{2}\arbrebaLabc, f_4=\frac{2}{3}\arbreaLabe+\arbreaLabc&\\
&T(f_4)= \arbreaLabe \,\,\slash\!\!\!\!\!\in T(J'), \hbox{ then: }&\\
&\phi_5=3 \arbrebaLabd + \arbrebbLabc-\frac{3}{2} \arbrebaLabc - 3\arbrebaLabd + \frac{1}{3}\arbreaLabe+\arbreaLabf,\,\,h_5= \frac{3}{2}\arbrebaLabc+\frac{2}{3}\arbreaLabe, f_5=\arbreaLabc&\\
&T(f_5) = \arbreaLabc \in T(J'), \hbox{ choose } g_5=\frac{1}{2}(\arbreaLabb+2\arbreaLabc) \in J'&\\
&\phi_6=3 \arbrebaLabd + \arbrebbLabc-\frac{3}{2} \arbrebaLabc - 3\arbrebaLabd + \frac{1}{3}\arbreaLabe+\arbreaLabf+\frac{1}{2}\arbreaLabb+\arbreaLabc,\,\,h_6= \frac{3}{2}\arbrebaLabc+\frac{2}{3}\arbreaLabe, f_6=-\frac{1}{2}\arbreaLabb&
\end{align*}
\begin{align*}
&T(f_6)= \arbreaLabb \,\,\slash\!\!\!\!\!\in T(J'), \hbox{ then: }&\\
&\phi_7=3 \arbrebaLabd + \arbrebbLabc-\frac{3}{2} \arbrebaLabc - 3\arbrebaLabd + \frac{1}{3}\arbreaLabe+\arbreaLabf+\frac{1}{2}\arbreaLabb+\arbreaLabc,\,\,h_7= \frac{3}{2}\arbrebaLabc+\frac{2}{3}\arbreaLabe -\frac{1}{2}\arbreaLabb, f_7=0,&
\end{align*}
\noindent then we obtain that $Can(f,J')= \frac{3}{2}\arbrebaLabc+\frac{2}{3}\arbreaLabe -\frac{1}{2}\arbreaLabb$\,.\\

One can note that choosing different $g$'s at each step in the procedures above while changing the intermediate computations would not change the final result. 
\end{exam}

Theorem \ref{span} does not describe the contents of each of $T(I)$\label{Tf4} and $O(I)$. We try here to get a description of them, using the magma of planar rooted trees $T^{E}_{pl}$ with its $K$\label{not8-K}-linear span $\Cal{T}^{E}_{pl}$. Let $J$ be the (two-sided) ideal of $\Cal{T}^{E}_{pl}$\label{E30} generated by the pre-Lie identity with respect to the magmatic product $\lbutcher$. By Theorem \ref{span}, we have:
\begin{equation}
\Cal{T}^{E}_{pl} = J \oplus S\!pan_K(O(J)).
\end{equation} 

\begin{prop}\label{princ}
$O(J)$ is the set of $\sigma \in T^E_{pl}$ such that for any $v \in V(\sigma)$, the branches starting from $v$ are displayed in nondecreasing order from left to right.
\end{prop}

\noindent The following lemma will help us to prove Proposition \ref{princ}:

\begin{lem}\label{non-increasing}
Any tree $\sigma$ in $T^E_{pl}$ which does not verify the condition of Proposition \ref{princ} belongs to $T(J)$.
\end{lem} 
\begin{proof}
Let $\sigma= B_{+,_,r}(\sigma_1 \cdots \sigma_k)$ be a tree in $T^E_{pl}$, with $k$ branches for $k \geq 2$ starting from the root, such that $\sigma_{i-1} \succ \sigma_i, \hbox{ for some } i= 1, \ldots, k-1$. We find that:

\begin{equation}\label{elementf}
f= \arbreLabgencasel \;\;\;\;- \;\;\;\;\arbreLabcaseji\;\;\;\;+ \;\;\;\; \arbreLabcaseij \;\;\;\; - \;\;\;\;\arbreLabgencasell\,
\end{equation}
is an element in $J$ such that $T(f)= \sigma$. If the branches start from a vertex $v$ different from the root, the subtree $\sigma_v$, obtained by taking $v$ as a root, is a factor of the tree $\sigma$. It is easily seen that $\sigma$ is the leading term of the element $f \in J$ obtained by replacing the factor $\sigma_v$ by the corresponding factor given by \eqref{elementf}.     
\end{proof}
\noindent As a consequence of Lemma \ref{non-increasing}, we get the following natural result.
\begin{cor}\label{interm}
$O(J)$ is contained in the set $\big\{ \sigma \in T^E_{pl} : \sigma \hbox{ has non decreasing branches} \big\}$.
\end{cor}

\begin{proof}[Proof of Proposition \ref{princ}]

Using the graduation of $T^E_{pl}$, with respect to the degree of trees therein, there is a one-to-one bijection between the subset $\big\{ \sigma \in T^E_{pl} : \sigma \hbox{ has non decreasing branches} \big\}_n$ and the the homogeneous component $T^E_n$ of all $E$-decorated (non-planar) rooted trees of degree $n$, i.e.: 
$$\# \big\{ \sigma \in T^E_{pl} : \sigma \hbox{ has non decreasing branches} \big\}_n = \# T^E_n, \hbox{ for all } n \geq 1.$$

But, $O(J)_n \tilde{=} T^E_n, \hbox{ for all } n \geq 1$, have the same cardinality, hence coincide according to Corollary \ref{interm}: 
\begin{equation}\label{intermi}
O(J) = \big\{ \sigma \in T^E_{pl} : \sigma \hbox{ has non decreasing branches} \big\}.
\end{equation}
This proves the Proposition \ref{princ}.   
\end{proof}

\noindent In the next Theorem, we describe the set $O(J')$ for the ideal $J'$ defined above by \eqref{ideal}.

\begin{thm}\label{main}
The set $O(J')$ is a set of ladders, or equivalently, the magmatic ideal $T(J')$\label{Tf5} contains all the trees which are not ladders.
\end{thm} 
\begin{proof}
We use here the induction on the number $n$ of vertices. Let $\sigma$ be a tree in $T^E_{pl}$, which is not a ladder, with $k$ branches (starting from the root) and $n$ vertices. Since $\sigma$ is not a ladder, then $n$ must be greater than or equal to $3$. If $n=3$, and $k=1$ then $\sigma$ is a ladder. Hence, for $k=2$, we have that:\\\\
$\sigma = \arbrexyLabb \hbox{ is an element of $T(J')$, since there is } f= |x| \arbrexyLabb + \big( |y| + |r| \big) \arbrexylLab \hbox{ in $J'$, such that } T(f) = \sigma,$
for any $x, y, r \in E$. Also, for any $\tau \in T^E_{pl}$, the elements $\sigma \lbutcher \tau$ and $\tau \lbutcher \sigma$ are in $T(J')$ (since $T(J')$ is an ideal).\\

Suppose that any (no-ladder) tree in $T^E_{pl}$\label{E31} with $q$ vertices, where $q < n$, is an element in $T(J')$, let $\widetilde{\sigma} \in T^{E}_{\!pl}$ with $n$ vertices and $k$ branches, which is not a ladder, then:
\begin{enumerate}
\item If $k=1$, the tree $\widetilde{\sigma}$ is written $\sigma \lbutcher \racineLabr $, where $\sigma$ is not a ladder. Then $\sigma \in T(J')$ by the induction hypothesis, hence $\widetilde{\sigma} \in T(J')$ because $T(J')$ is an ideal.
\item \label{case ii for k}The case $k=2$. This corresponds to the case $\widetilde{\sigma} = \sigma \lbutcher l_m$ , where $l_m$ is a ladder in $T^E_{pl}$, with $m$ vertices for $m\geq 2$. If $\sigma$ is an element of $T(J')$ then so is $\widetilde{\sigma}$. If not, $\sigma$ is a ladder by the induction hypothesis. See the discussion below.
\item The case $k \geq 3$. These are trees $\widetilde{\sigma} = \big(\sigma \lbutcher \tau \big)$ where $\tau \in T^E_{pl}$, with $k-1$ branches, is not a ladder. We have then $\widetilde{\sigma} \in T(J')$ by induction hypothesis.
\end{enumerate}
  
Let us discuss the case \eqref{case ii for k} when $\sigma$ is a ladder and the ladder $l_m$ does not belong to $T(J')$. Let $l_{1}, l_{2}$ be ladders in $T^E_{pl}$ with $n_1, n_2$ vertices respectively, where $n_1, n_2 < n$, and let:
\begin{equation}\label{definition of sigma}
\widetilde{\sigma} =\arbrelmlLab = l_{1} \lbutcher ( l_{2} \lbutcher \racineLabr),\,\,\sigma' =\arbrelplLab = l_{2} \lbutcher ( l_{1} \lbutcher \racineLabr). 
\end{equation}
By the pre-Lie identity, with respect to the left Butcher product $\lbutcher$\label{lb5}, we find the following element:

\begin{equation}\label{element f0}
f_0= \arbrelmlLab - \arbrelmlpLab +  \arbrelplmLab - \arbrelplLab 
\end{equation}
in $J'$, such that $\widetilde{\sigma}, \sigma'$ are bigger trees, with respect to the order $\preceq$ defined in \eqref{order}, than the two other trees in $f_0$. Let $|l_{i}| = p_i, \hbox{ where } p_i > 0$, for $i=1, 2$. We have the following cases for $p_i$: 
\begin{enumerate}
\item Either $p_1 = p_2$, then in this case we take the elements:
\newline

\begin{equation}\label{element g}
 g = p_2  \arbrelplLab + (p_1 + |r|)\!\!\!\arbrelmrLab\,\,,\,f_1 = \arbrelmrlpLab - \arbrelmrlprLab  + \arbrelplmrrLab - \arbrelprlmLab\,,
\end{equation}
in $J'$, where $l_{2} = l^{(1)}_{2} \lbutcher \racineLabrrr$. Then we get the element:
\begin{equation}\label{element f}
f= p_2 f_0 + g - (p_1+ |r| ) f_1 \in J',
\end{equation} 
such that $T(f)= \widetilde{\sigma}$, since:
\newline
$$\arbrelprlmLab\,\prec\,\arbrelmlLab, \hbox{ for the order } \preceq.$$
 
\item Or, $p_2 < p_1$, then $\widetilde{\sigma} = T(f_0)$\label{Tf6}, where $f_0$ is the element described in \eqref{element f0}, hence $\widetilde{\sigma} \in T(J')$. 

\item\label{case of sigma} Or, $p_1 < p_2$, here we have that $\widetilde{\sigma} \prec \sigma'$ and the element $f_0$ described in \eqref{element f0} is an element in $J'$ such that $T(f_0) = \sigma'$, hence $\sigma' \in T(J')$. Now, for $\widetilde{\sigma}$ we can get an element in $J'$ such that $\widetilde{\sigma}$ becomes the leading term of this element, as follows: we replace the tree $\sigma' = l_{2} \lbutcher ( l_{1} \lbutcher \racineLabr)$ in $f_0$ by the tree:

\begin{equation}
 \sigma'':= \big(l_{1} \lbutcher \racineLabr\big) \lbutcher l_{2} = \arbrelmrLab\, ,
\end{equation}
using the element $g$ described in \eqref{element g}. This new tree $\sigma''$ is also greater than $\widetilde{\sigma}$ with respect to the order $\preceq$. By the pre-Lie identity, we can get the element $f$ described in \eqref{element f} such that\label{lb6}:

$$\widetilde{\sigma} \hbox{ and } \sigma'_1:=l^{(1)}_{2} \lbutcher \big( \big(l_{1} \lbutcher \racineLabr\big)\lbutcher \racineLabrrr\,\big) = \arbrelprlmLab \hbox{ are the two biggest trees appearing in this element. }$$

We verify whether $p_1= |l_{1}| > |l_{2}^{(1)}| = p_2 - |r_1|$, i.e. $\sigma'_1 \preceq \widetilde{\sigma}$, or not. If so, then $\widetilde{\sigma} \in T(J')$. If not, we replace $\sigma'_1$ in $f$ by the tree:\\

\begin{equation}\label{anti-symmetry ii}
 \sigma''_{1}:= \big( \big(l_{1} \lbutcher \racineLabr\big)\lbutcher \racineLabrrr\,\big) \lbutcher l^{(1)}_{2} = \arbrelmrrlpLab\;\;.
\end{equation}

If $n_2=1$, the tree $\sigma''_{1}$ is a ladder. If $n_2 \geq 2$, then $\sigma''_{1}$ is not a ladder and is greater than $\widetilde{\sigma}$. Then we need to apply the pre-Lie identity once again to the tree $\sigma''_{1}$ in \eqref{anti-symmetry ii}, and replace it by:\\

$$ \sigma'_2 := l^{(2)}_{2} \lbutcher \big( \big( \big(l_{1} \lbutcher \racineLabr\big)\lbutcher \racineLabrrr\,\big) \lbutcher \racineLabrrrr\,\big) = \arbrelplmrrrLab, \hbox{ where } l_{2}^{(2)} \lbutcher \racineLabrrrr = l_{2}^{(1)} . $$

Let $p_{2}^{(i)} = |l_{2}^{(i)}|$, where $l_{2}= (\cdots((l_{2}^{(i)} \lbutcher \racineLabri\;) \lbutcher \racineLabrii\;\;) \cdots) \lbutcher \racineLabrrr, \hbox{ for } i \geq 1$. After a finite number $s$ of steps  applying the pre-Lie identity in the expression:\\

\begin{equation}
\sigma''_{s }:= \big(\big(\cdots \big((l_{1} \lbutcher \racineLabr) \lbutcher \racineLabrrr\;\big) \cdots \big) \lbutcher \racineLabrjj\;\;\;\big) \lbutcher \big(l_{2}^{(s)} \lbutcher \racineLabrj\;\big) = \arbrelmrslp,\, \hbox{ where } \sigma'_s =  \arbrelplmrs
\end{equation}
\newline
which can be formulated as: 

\begin{equation}\label{recursive}
f_s= \arbrelmrslp - \arbrelplmrrrsLab + \arbrelmrrlpsLab - \arbrelplmrs \in J',
\end{equation}
we can find an element $f \in J'$, such that $\widetilde{\sigma}$ and $\sigma'_s$ become bigger trees of $f$ with $p_2^{(s)} < p_1$, i.e. $\sigma'_s \prec \widetilde{\sigma}$. Hence, $\widetilde{\sigma}$ described in \eqref{definition of sigma} is in $T(J')$\label{Tf7}. Then, Theorem \ref{main} is proved.  
\end{enumerate}
\end{proof}
\section{A monomial basis for the free Lie algebra}\label{CIIIsecIII}
The set $T^ E$ forms the free Non-Associative Permutive (NAP) magma generated by the set $\{ \racineLabaa : \hbox{ for } a \in E \}$\label{E32}, under the usual Butcher product $\butcher$. Corresponding to the total order defined in \eqref{order}, we can define a non-planar version $\preceq$ of this order, as follows:

\begin{equation}\label{non-planar version}
\hbox{ for any $s, t \in T^{E}$, then} \,s \preceq t \hbox{ if and only if } 
\end{equation}
\begin{enumerate}
\item $|s| < |t|, \hbox{ or }$:
\item $|s| = |t| \hbox{ and } b(s) <  b(t),$ or:
\item\label{conditioniii} $|s| = |t|, b(s) =  b(t) = k$ and $s= B_{+, r}(s_1 \ldots s_k), t=B_{+, r'}(t_1 \ldots t_k)$ such that $\exists\, j \leq k$, with $s_i = t_i$, for $i < j, s_j \preceq t_j$ where $s_1 \preceq \cdots \preceq s_k, t_1 \preceq \cdots \preceq t_k$ are the branches of $s, t$ respectively, or:
\item  $|s| = |t|, b(s) =  b(t) = k$, $s_l = t_l , \hbox{ for all } l= 1, \ldots, k$ and $r \leq r'$, where $\racineLabr$ (respectively $\racineLabrr$) is the root of $s$ (respectively $t$).
\end{enumerate}
 
By the same way as in Proposition \ref{well order}, we observe that the order $\preceq$ defined in \eqref{non-planar version} is a monomial well order. The space $\Cal{T}^E$ forms with the Butcher product the free NAP algebra generated by $E$ \cite{M.L06}. The first author introduced in \cite{AM2014}\footnote{See subsection 2.2 in \cite{AM2014}.} a section $S$ from the NAP algebra $(\Cal{T}^ E, \butcher)$ into the magmatic algebra $(\Cal{T}^ E_{pl}, \lbutcher)$:
\diagrama{
\xymatrix{
( \Cal {T}^{E}_{pl}, \lbutcher ) \ar@{->>}[r]^{\pi}& ( \Cal{T}^ E, \butcher ) \ar@<02mm>@{.>}[l]^S. }}

Here, we choose $S(t) = S_{min}(t) := Min_{\preceq} \big\{ \tau \in T^{E}_{pl}: \pi(\tau)=t \big\}$, for any $t \in T^E$, where $ Min_{\preceq} \{ -\}$ means that we choose the minimal element $\tau$ in $T^E_{pl}$ with respect to the order ''$\preceq$'' with $\pi(\tau) = t$. 
\begin{prop}\label{section s increasing}
The section map $S_{min}$ defined above is an increasing map.
\end{prop}
\begin{proof}
Take two trees $s$ and $t$ in $T^E$\label{E33} with $s \preceq t$.  The section $S_{min}$,  obviously, respects the degree and the number of branches of the trees. Hence, we can suppose $|s| = |t|$ and $b(s) = b(t) = l$. We have then:

\begin{equation}
s =B_{+, r}(s_1, \ldots, s_l),\; t = B_{+, r'}(t_1, \ldots, t_l), \hbox{ with } s_1 \preceq \cdots \preceq s_l\;,\; t_1 \preceq \cdots \preceq t_l.
\end{equation}
Condition \eqref{conditioniii} of the definition, in \eqref{non-planar version}, of the order $\preceq$ exactly means that the $l$-tuple of branches of $ S_{min}(s)$ is lexicographically smaller than the $l$-tuple of branches of $ S_{min}(t)$. If $s$ and $t$ have the same branches and $s \preceq t$, we also have $ S_{min}(s) \preceq  S_{min}(t)$, as one can see by comparing the roots. This proves Proposition \ref{section s increasing}. 
\end{proof}

\begin{prop}
The section map $S_{min}$ on $T^E$ is a bijection onto $O(J)$\label{Tf8}, where $J$ is the (two-sided) ideal generated by the pre-Lie identity in  $\big(\Cal{T}^E_{pl}, \lbutcher\big)$.
\end{prop}
\begin{proof}
Clear from Proposition \ref{princ}.
\end{proof}
\noindent Define a relation $R$ on $T^E$ as follows:
\begin{equation}\label{relation R}
 sRs' \hbox{ if and only if there are } t, t' \in T^E \hbox{ and } v, w \in V(t') \hbox{ such that } s = t \to_v t', s'= t \to_w t'
\end{equation}
for $s, s' \in T^E$, and $w$ is related with $v$ by an edge $\arbreaaLab$ with $w$ above $v$. Let $\#$  be the transitive closure of the relation $R$ defined in \eqref{relation R}, i.e. for $s, s' \in T^E$, we say that $s \#  s'$ if and only if there is $s_1, \ldots, s_l \in T^E$ such that $s R s_1 R \ldots R s_l R s'$.
\begin{lem}\label{deduced}
Let $s, s', t \in T^E$, if $s' \preceq s$ then $s' \to_v t \preceq s \to_v t, \hbox{ for } v \in V(t)$.

\end{lem}
\begin{proof}
Immediate from the definition \eqref{non-planar version} of the order $\preceq$.
\end{proof} 
\begin{lem}
Let $s, s' \in T^E$, if  $s \#  s'$ then $s' \prec s$.
\end{lem}
\begin{proof}
For $s, s' \in T^E$\label{np1}, if $s R s'$, then by definition of the relation $R$ in \eqref{relation R}, there are $t, t' \in T^E$ and $v, w \in V(t)$ such that $s= t \to_v t',\,s' = t \to_w t'$, and an edge  $\arbreaaLab$ in $t'$. Obviously, the tree obtained by grafting $t$ on the tree $t'$ at $v$ is greater, with respect to the order $\preceq$, than the tree deduced by grafting $t$ on $t'$ at $w$, i.e $s' \prec s$. The passage from $R$  to \#  is obvious.
\end{proof}

\begin{prop}\label{compatibility}
The Butcher product $\butcher$ is compatible with the relation $R$, i.e. for $s, s', t \in T^E$, if $s R s'$ then $(s \butcher t )R (s' \butcher t) \hbox{ and } (t \butcher s) R (t \butcher s')$. Also, if $s R s' \hbox{ and } t R t'$ then $(s \butcher t ) \#  (s' \butcher t'), \hbox{ for } t' \in T^E$.

\end{prop}
\begin{proof}
For any $s, s', t, t' \in T^E$, if $sRs' \hbox{ and } tRt'$, then by definition of $R$ we have:
\begin{align*}
& s= s_1 \to_v s_2, s' = s_1 \to_w s_2, \hbox{ for } v, w \in V(s_2), \hbox{ with } \arbreaaLab \hbox{ in } s_2, \hbox{ and } t= t_1 \to_u t_2, t' = t_1 \to_{w'} t_2,\\
& \hbox{ for } u, w' \in V(t_2), \hbox{ with } \arbreabLab \hbox{ in } t_2. \\
&\hbox{ Let: } s \butcher t = (s_1 \to_v s_2) \butcher (t_1 \to_u t_2) = s_1 \to_v s'', \hbox{ for } v \in V(s''), \hbox{ where } s''= s_2 \butcher t, \hbox{ and }\\
& s' \butcher t = (s_1 \to_w s_2) \butcher (t_1 \to_u t_2) = s_1 \to_w s'', \hbox{ for $\arbreaaLab$ in } s'', \hbox{ then: }
\end{align*}
\begin{equation}\label{st}
s \butcher t = (s_1 \to_v s'') R  (s_1 \to_w s'' )= s' \butcher t.
\end{equation}
Also, for $s' \butcher t' = (s_1 \to_w s_2) \butcher (t_1 \to_{w'} t_2) = t_1 \to_{w'} s''', \hbox{ where } s''' = s' \butcher t_2, \hbox{ with } w' \in V(t_2) \subset V(s''')$, and $s' \butcher t = (s_1 \to_w s_2) \butcher (t_1 \to_{u} t_2) = t_1 \to_{u} s''', \hbox{ for } u \in V(s''')$. Then we have:
\begin{equation}\label{s't}
s' \butcher t R s' \butcher t'.
\end{equation}
One can verify that $s \butcher t R s \butcher t'$ by following the same steps as above. So, from \eqref{st} and \eqref{s't}, we obtain that $s \butcher t  \#  s' \butcher t'$. 
\end{proof}
For any $t \in T^E$\label{E34}, define a class of $t$ with respect to $\#$  by:
\begin{equation}
[t]_{\#}:= \{s \in T^E: t \#  s \}.
\end{equation}
This class has the following properties:
\begin{enumerate}
\item $t$ is maximal among the representative elements in the class $[t]_{\#}$, i.e. for any $s \in [t]_{\#}$ then $s \preceq t$. This property is deduced from Lemma \ref{deduced}.
\item For any $s \in [t]_{\#}$, then $[s]_{\#}\subset [t]_{\#}$. 
\end{enumerate}

\begin{lem}
For any $t \in T^E$, then:
\begin{equation}\label{realized}
\widetilde{\Psi}_{S_{min}}(t)= \sum_ {s \in [t]_{\#}}\beta_{S_{min}}(s,t)s,
\end{equation}

\begin{figure}[h]
\vspace{-35pt}
\centering
\diagrama{
\xymatrix{
( \Cal {T}^{E}_{pl}, \lbutcher ) \ar[r]^{\Psi}\ar@<02mm>@{->>}[d]^{\pi}
&( \Cal {T}^{E}_{pl}, \searrow ) \ar@{->>}[d]^{\pi}\\
  ( \Cal{T}^ E, \butcher ) \ar@{.>}[r]_{\widetilde{\Psi}_{S_{min}}}\ar@<03mm>@{.>}[u]^{S_{\!min}} & ( \Cal{T}^ E, \to)}}
\caption{}
\label{}
\end{figure}
\noindent where the map $\widetilde{\Psi}_{S_{min}}$ and the coefficients $\beta_{S_{min}}(s,t)$ are described in \cite[Corollary 5]{AM2014}.  
\end{lem}
\begin{proof}
We prove this Lemma by the induction on the number of vertices of  the tree. Suppose that \eqref{realized} is realized for any tree in $T^E$\label{E35} with a number of vertices less than or equal to $n$. Take $t \in T^E$ be a tree, such that $\#V(t) = n+1$ and $t= t_1 \butcher t_2$, where $t_1$ is the minimal branch of $t$ with respect to the order $\preceq$. Then we have:
\begin{align*}
\widetilde{\Psi}_{S_{min}}(t) &= \widetilde{\Psi}_{S_{min}}(t_1 \butcher t_2) &\\
&= \overline{\Psi} \circ S_{min}(t_1 \butcher t_2) &\\
&= \overline{\Psi} \big(S_{min}(t_1)  \lbutcher S_{min}(t_2) \big)&\\
&= \pi \big( \Psi \circ S_{min}(t_1)  \searrow \Psi \circ S_{min}(t_2) \big)&\\
&= \widetilde{\Psi}_{S_{min}}(t_1) \to \widetilde{\Psi}_{S_{min}}(t_2)&\\
&= \left(\sum_ {s' \in [t_1]_{\# }} \beta_{S_{min}}(s',t_1)s'\right) \to  \left( \sum_ {s'' \in [t_2]_{\# }} \beta_{S_{min}}(s'',t_2)s''\right)&\\
&= \sum_{{s' \in [t_1]_{\# }\atop {s'' \in [t_2]{\# }}}} \beta_{S_{min}}(s',t_1)\,\beta_{S_{min}}(s'',t_2)\; s' \to s''\,.&
\end{align*}
From Proposition \ref{compatibility}, we have that\label{lb7}:
\begin{equation}\label{defs}
t= t_1 \butcher t_2\,\# \,s := \,s'\butcher s''\,\#  \,s' \to_{v} s'', \hbox{ for } v \in V(s'').
\end{equation}

Let $s^v$ be the smallest branch of the tree $s$, defined above in \eqref{defs}, starting from $v$, and $s_v$ be the corresponding trunk (what remains when the branch $s^v$ is removed). Then we have:
\begin{equation}\label{target}
\beta_{S_{min}}(s,t) = \sum_{v \in V(s)}\beta_{S_{min}}(s^v,t_1)\,\beta_{S_{min}}(s_v,t_2).
\end{equation}
The formula \eqref{target} above is induced by the formula $2.15$ and the definition of the coefficients $\beta_{S_{min}}(s,t)$ described in \cite{AM2014}. Hence, we get:
$$\widetilde{\Psi}_{S_{min}}(t)= \sum_ {s \in [t]_{\# }} \beta_{S_{min}}(s,t)s.$$
\end{proof}

\begin{cor}\label{term of psi}
Let $t \in T^E$, then the maximal term $T\big(\widetilde{\Psi}_{S_{min}}(t)\big)$\label{Tf9}, with respect to the order defined in \eqref{non-planar version}, of $\widetilde{\Psi}_{S_{min}}(t)$ is the tree $t$ itself. 
\end{cor}

From \cite{AM2014}, we have that the set $\Cal{B} = \left\{ \widetilde{\Psi}_{S_{min}}(t): t \in T^E \right\}$ forms a monomial basis for the free pre-Lie algebra $\big(\Cal{T}^E, \to \big)$. Let $I$ be the (two-sided) ideal generated by the elements on the form described in \eqref{def of I}, then we have the following commutative diagram:\\

\begin{figure}[h]
\centering{}
\diagrama{
\xymatrix{
\big(\Cal{T}^E, \to \big) \ar@{->>}[r]^{q} \ar@{->>}[d]^{\Phi} & \big( \Cal{L}'(E), \rhd^{\!*} \big) \ar@{->}[dl]|{\,\widetilde{=}\,}\\ \big( \Cal{L}(E), \rhd \big)}}
\caption{}
\label{}
\end{figure}
\newpage
\noindent where $\Cal{L}'(E)=\Cal{T}^E / I$\label{CL6}, and the product $\rhd^{\!*}$ is the pre-Lie product defined in \eqref{pre-lie}. $\Cal{L}(E)$ is the free Lie algebra generated by $E$ which carries the pre-Lie algebra structure by the product $\rhd$ defined in \eqref{rhd}. The restriction of $\Phi$  to $S\!pan_K(O(I))$\label{not9-K} is an injective map\label{lb8}. Indeed, for any $h_1, h_2, \in S\!pan_K(O(I))$, 
\begin{align*}
&\Phi(h_1) = \Phi(h_2)&\\
\Rightarrow\;\;&\Phi(h_1 - h_2) = 0&\\
\Rightarrow\;\;&(h_1 - h_2) \in Ker\,\Phi = I&\\
\Rightarrow\;\;&(h_1 - h_2) \in S\!pan_K(O(I)) \cap I = \big\{ 0 \big\}&\\
\Rightarrow\;\;& h_1 - h_2 =  0 &\\
\Rightarrow\;\;& h_1  = h_2.&
\end{align*}

Also, since $\Phi: \Cal{T}^E \longrightarrow \Cal{L}(E)$ is a surjective map, then we have:
\begin{align*}
\Cal{L}(E) &= \Phi\big(\Cal{T}^E\big)&\\
&= \Phi\big( I \oplus S\!pan_K(O(I)) \big) \;\;\;\;(\hbox{ by Theorem \ref{span} })&\\
&= \Phi\big(S\!pan_K(O(I))\big)\;\;\;\;\;\;\;\;,\hbox{ since } Ker\,\Phi = I \hbox{ and } \Phi\big(I\big) = \big\{0\big\}.& 
\end{align*}  

Hence, $\Phi: S\!pan_K(O(I)) \longrightarrow \Cal{L}(E)$ is a surjective and an injective map. Then it is an isomorphism of vector spaces.
   
\begin{thm}\label{basis for L(E)}
For any $t \in O(I)$, we have:
\begin{equation}\label{POEO}
\widetilde{\Psi}_{S_{min}}(t) = Can\big(\widetilde{\Psi}_{S_{min}}(t), I\big) = t.
\end{equation}
Moreover, the set $\widetilde{\Cal{B}} := \left\{ \Phi(t) : t \in O(I) \right\}$ is a monomial basis for the pre-Lie algebra $\big( \Cal{L}(E), \rhd \big)$. 
\end{thm}

\begin{proof}
The property \eqref{POEO} is induced from Theorem \ref{main} and the definition of $\widetilde{\Psi}_{S_{min}}$. We obviously have that the set $\Cal{B}' = O(I)$ is a basis for $S\!pan_K(O(I))$. Therefore, as $\Phi: S\!pan_K(O(I)) \longrightarrow \Cal{L}(E)$ is an isomorphism of vector spaces, $\widetilde{\Cal{B}} := \Phi(\Cal{B}')$ forms a basis for the pre-Lie algebra $\big(\Cal{L}(E), \rhd \big)$. This basis is monomial thanks to \eqref{POEO}, such that:
$$\Phi(t) = \Phi\big( \widetilde{\Psi}_{S_{min}}(t)\big), \hbox{ for all } t \in O(I),$$ 
this proves Theorem \ref{basis for L(E)}.      
\end{proof}  

\noindent Consequently, we get the following immediate result.
\begin{cor}\label{MBOL}
The set $\widetilde{\Cal{B}} := \left\{ \Phi(t) : t \in O(I) \right\}$\label{CL7} is a monomial basis for the free Lie algebra $\big( \Cal{L}(E),[\cdot, \cdot] \big)$.
\end{cor}

\begin{exams}
Here,we calculate few first bases $\widetilde{\Cal{B}}_n$ for homogeneous components $\Cal{L}_n$ of the free Lie algebra $\Cal{L}(E)$ up to $n=4$, using Corollary \ref{MBOL}, as follows:

\begin{enumerate}
\item\label{g. case} As a particular case, take $E = \{a_i : i \in \mathbb{N}\}$, such that $ |a_i| = i$, for all $i \in \mathbb{N}$, with total order $a_1 < a_2 < \cdots < a_s < \cdots$ on the generators. From \cite{AM2014}, we have:

\begin{align*}
\Cal{B}\big(\Cal{T}^{E}_1\big) &= \left\{ \racineLaba \;: a_1 \in E\label{E36} \right\} \,.&\\
\Cal{B}\big(\Cal{T}^{E}_2\big) & = \left\{ \racineLabaii \;: a_2 \in E\right\}  \sqcup \left\{ \arbreaoLaba \;: a_1 \in E \right\}.&\\
 \Cal{B}\big(\Cal{T}^{E}_3\big) & = \left\{ \racineLabooo \;: a_3 \in E \right\} \sqcup \left\{ \arbreaLabooo\,,\,\arbreaLaboo \;: a_1, a_2 \in E \right\} \sqcup \left\{ \arbrebaLabo\;,\;\; \arbrebbLabo + \arbrebaLaboi \;: a_1 \in E \right\}.&\\
\Cal{B}\big(\Cal{T}^{E}_4\big) & = \left\{ \racineLaboooo \;: a_4 \in E \right\} \sqcup \left\{ \arbreaLabto\,,\,\arbreaLabot\,,\,\arbreaLabtt \;: a_1, a_2, a_3  \in E \right\} \sqcup &\\
&\;\;\;\;\;\;\left\{ \arbrebaLaboot\;,\;\arbrebaLabtoo\;,\;\arbrebaLaboto\;,\;\; \arbrebbLabtoo + \arbrebaLaboot \;\;,\;\; \arbrebbLaboto + \arbrebaLabtoo\;: a_1, a_2 \in E\right\} \sqcup &\\
&\;\;\;\;\;\;\left\{  \arbrecaLabo\;\;,\; \arbrecbLabo + \arbrecaLabo\;\;\,,\;\, \arbreccLabo \;+ \arbrecbLabo + \arbrecaLabo\;\;\,,\;\,\arbrecorLabo\;+\; 3 \; \arbreccLabo \;+ \arbrecbLabo + \arbrecaLabo\;\;: \,a_1 \in E \right\}.&
\end{align*}
Then, we get the following monomial bases $\widetilde{\Cal{B}}_n$ for $\Cal{L}_n$\,, up to $n=4$:

\begin{align*}
\widetilde{\Cal{B}}_1 &= \{ a_1 \}\,.&\\\\
\widetilde{\Cal{B}}_2 &= \{ a_2 \}\,.&\\\\
\widetilde{\Cal{B}}_3 &= \big\{ a_3, [a_1, a_2] \big\}\,.&\\\\
\widetilde{\Cal{B}}_4 &= \{ a_4, [a_1, a_3], \big[[a_1, a_2], a_1\big] \big\}\,.&
\end{align*}

\item\label{s. case} Let us take $E=\{x, y \}$ ordered by $x<y$, such that $ |x| = |y| = 1$. Denote by $\racinecirc$ the vertex decorated by $x$, and $\racine$ the vertex decorated by $y$, such that $ \racinecirc < \racine$\,. Using the order defined in \eqref{non-planar version}, we arrange the first terms of $T^E$ as follows: 
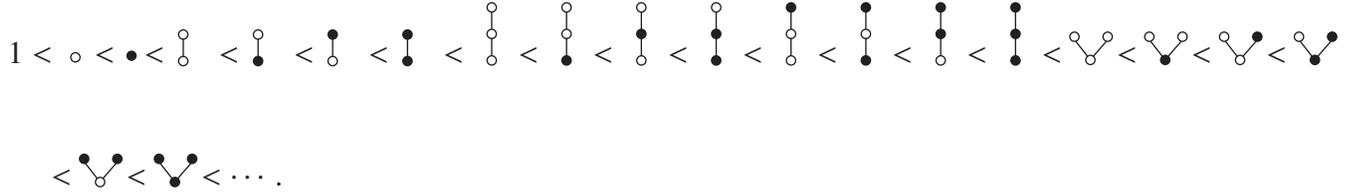
\begin{align*}
&1<\racinecirc<\racine<\arbreacirca<\arbreacircb<\arbreacircc<\arbreacircd<\arbrebacirc<\arbrebbcirc<\arbrebccirc<\arbrebdcirc<\arbrebecirc<\arbrebfcirc<\arbrebgcirc<\arbrebhcirc<\!\!\!\!\!\!\!\!\!\!\arbrecacirc\!\!\!\!\!\!\!\!\!\!<\!\!\!\!\!\!\!\!\!\!\arbrecbcirc\!\!\!\!\!\!\!\!\!\!<\!\!\!\!\!\!\!\!\!\!\arbrecccirc\!\!\!\!\!\!\!\!\!\!<\!\!\!\!\!\!\!\!\!\!\arbrecdcirc&\\
&\;\;\;\;<\!\!\!\!\!\!\!\!\!\!\arbrecgcirc\!\!\!\!\!\!\!\!\!\!<\!\!\!\!\!\!\!\!\!\!\arbrechcirc\!\!\!\!\!\!\!\!\!\!<\cdots.&
\end{align*}
Also, we calculate here the monomial bases for the homogeneous components $\Cal{T}^{E}_{\!n}$\label{E37} up to $n=4$:

\begin{align*}
\Cal{B}\big(\Cal{T}^{E}_1\big) &= \Big\{ \racinecirc, \racine \,\Big\} \,.&\\
\Cal{B}\big(\Cal{T}^{E}_2\big) & = \Big\{ \arbreacirca\, ,\,\arbreacircb\,,\,\arbreacircc\,,\, \arbreacircd \Big\}.&\\
 \Cal{B}\big(\Cal{T}^{E}_3\big) & = \Big\{ \arbrebacirc\,,\,\arbrebbcirc\,,\,\arbrebccirc\,,\, \arbrebdcirc\,,\,\arbrebecirc\,,\, \arbrebfcirc\,,\, \arbrebgcirc\,,\, \arbrebhcirc, \!\!\!\!\!\!\!\!\!\!\arbrecacirc \!\!\!\!\!\!\!\!\!\!+ \arbrebacirc,\!\!\!\!\!\!\!\!\!\!\arbrecbcirc\!\!\!\!\!\!\!\!\!\!+ \arbrebbcirc, \!\!\!\!\!\!\!\!\!\!\arbrecccirc\!\!\!\!\!\!\!\!\!\!+ \arbrebccirc, \!\!\!\!\!\!\!\!\!\!\arbrecdcirc\!\!\!\!\!\!\!\!\!\!+ \arbrebdcirc, \!\!\!\!\!\!\!\!\!\!\arbrecgcirc\!\!\!\!\!\!\!\!\!\!+ \arbrebgcirc, \!\!\!\!\!\!\!\!\!\!\arbrechcirc \!\!\!\!\!\!\!\!\!\!+ \arbrebhcirc \Big\}.
 \end{align*}
 \begin{align*}
\Cal{B}\big(\Cal{T}^{E}_4\big) & = \Big\{ \underbrace{\arbredacirc, \arbredaacirc, \ldots,   \arbredhhcirc, \arbredhhhcirc}_{\textrm{16 terms}}, \underbrace{\!\!\!\!\!\!\!\!\!\!\arbreeacirc\!\!\!\!\!\!\!\!\!\! + \arbredacirc, \!\!\!\!\!\!\!\!\!\!\arbreeaacirc\!\!\!\!\!\!\!\!\!\!+\arbredaacirc, \ldots,   \!\!\!\!\!\!\!\!\!\!\arbreehhcirc\!\!\!\!\!\!\!\!\!\!+\arbredhhcirc, \!\!\!\!\!\!\!\!\!\!\arbreehhhcirc \!\!\!\!\!\!\!\!\!\! + \arbredhhhcirc}_{\textrm{12 terms}} , & \\
&\;\;\;\;\;\underbrace{\;\arbreehcirc + \!\!\!\!\!\!\!\!\!\!\arbreeacirc\!\!\!\!\!\!\!\!\!\! + \arbredacirc, \;\arbreeehcirc +\!\!\!\!\!\!\!\!\!\!\arbreeaacirc\!\!\!\!\!\!\!\!\!\!+\arbredaacirc, \ldots, \arbreeeehcirc +\!\!\!\!\!\!\!\!\!\!\arbreehhcirc\!\!\!\!\!\!\!\!\!\!+\arbredhhcirc, \;\arbreeeeehcirc +\!\!\!\!\!\!\!\!\!\!\arbreehhhcirc \!\!\!\!\!\!\!\!\!\! + \arbredhhhcirc}_{\textrm{16 terms}},&\\
& \;\;\;\;\;\underbrace{\;\arbrefhcirc + 3 \;\arbreehcirc + \!\!\!\!\!\!\!\!\!\!\arbreeacirc\!\!\!\!\!\!\!\!\!\! + \arbredacirc, \ldots, \;\arbreffhcirc + 3 \;\arbreeeeehcirc +\!\!\!\!\!\!\!\!\!\!\arbreehhhcirc \!\!\!\!\!\!\!\!\!\! + \arbredhhhcirc  }_{\textrm{8 terms}} \Big\}.&
\end{align*}
Hence, we have:
\begin{align*}
\widetilde{\Cal{B}}_1 &= E\,.&\\\\
\widetilde{\Cal{B}}_2 &= \Big\{ [x, y] \;: x, y \in E \Big\}\,.&\\\\
\widetilde{\Cal{B}}_3 &= \Big\{ \big[[x, y], x\big] \;, \big[[x, y], y\big] \; : x, y \in E\Big\}\,.&\\\\
\widetilde{\Cal{B}}_4 &= \Big\{ \big[[[x, y], x], x\big]\,,\, \big[[[x, y], x], y\big]\,,\, \big[[[x, y], y], y\big]\;: x, y \in E \Big\}\,.&
\end{align*}
\end{enumerate}
\end{exams}
\begin{rmk}
In the monomial basis $\widetilde{\Cal{B}}_4$ for $\Cal{L}_4$\label{E38}, calculated in \eqref{s. case} above, we observe the following:

\noindent the tree $\arbredfcirc$ is not in $O(I)$, since there is an element $f= \arbreddcirc - \arbredfcirc\!$ that belongs to $I$ such that $T(f) = \arbredfcirc\!\!\!\in T(I)$. Indeed, from the pre-Lie identity, and the so-called weighted anti-symmetry identity described in \eqref{def of I}, we have, drawing non-planar trees explicitly:
\begin{align*}
&f_1=  \overline{\Psi} \Big (\big(\arbreacircb\!\!\lbutcher \racinecirc \big) \lbutcher \racine - \arbreacircb\!\!\lbutcher \arbreacircb - \big(\!\!\racinecirc \lbutcher \arbreacircb\!\!\big) \lbutcher \racine + \racinecirc \lbutcher \big(\arbreacircb\!\!\lbutcher \racine \big) \Big) =  \overline{\Psi} \Big( \arbreddcirc - \arbredhcirc - \arbredgcirc + \arbredicirc\,\Big),&\\ 
&f_2= \overline{\Psi} \Big( \big(\racinecirc \lbutcher \arbreacircb +2 \arbreacircb\!\!\lbutcher \racinecirc \big) \lbutcher \racine \Big) = \overline{\Psi} \Big( \arbredgcirc + 2 \arbreddcirc \Big), \hbox{ and } f_3=  \overline{\Psi} \Big (\racinecirc \lbutcher \arbrebdcirc + 3 \arbrebdcirc \lbutcher \racinecirc \Big) = \overline{\Psi} \Big(\,\arbredicirc + 3 \arbredfcirc \Big)&
\end{align*}
are elements in $I$, hence $f_4=f_1 + f_2 - f_3 = \overline{\Psi} \Big( 3 \arbreddcirc - 3 \arbredfcirc - \arbredhcirc\,\Big) \in I$\label{lb9}. But, $f_5= \overline{\Psi} \Big (\,\arbredhcirc\,\Big) \in I$, hence $f=f_4+f_5= 3 \arbreddcirc - 3 \arbredfcirc \in I.$ Then, we have:
\begin{align*}
\Phi(f) &= 3\,\big((x \rhd y)\rhd x\big) \rhd y - 3\,\big((x \rhd y)\rhd y\big) \rhd x -  (x \rhd y) \rhd (y \rhd x) &\\
 & = \big[[[x, y], x], y\big] - \big[[[x, y], y], x\big] + \big[[[x, y], [x, y]\big]&\\
& = 0\;,&
\end{align*}
and then,
$$\big[[[x, y], x], y\big] =  \big[[[x, y], y], x\big] .$$
\end{rmk}

\paragraph{\textbf{Acknowledgments}}{The authors thank K. Ebrahimi-Fard and L. Foissy for valuable discussions and comments.}


\end{document}